\documentclass[11pt]{amsart}
\usepackage{amssymb,latexsym,graphicx, amscd}
\usepackage{enumerate}
\usepackage{amsmath,amscd}
\usepackage{amsthm}
\usepackage{tikz}
\usepackage{url}

\newtheorem{theorem}{Theorem}[section]
\newtheorem{prop}[theorem]{Proposition}
\newtheorem{lemma}[theorem]{Lemma}
\newtheorem{remark}[theorem]{Remark}
\newtheorem{question}[theorem]{Question}
\newtheorem{definition}[theorem]{Definition}
\newtheorem{cor}[theorem]{Corollary}
\newtheorem{example}[theorem]{Example}

\newtheorem{theoremalt}{Question}[theorem]

\newenvironment{questionp}[1]{
  \renewcommand\thetheoremalt{#1}
  \theoremalt
}{\endtheoremalt}

\begin{document}

\title[Periodic symplectic and Hamiltonian diffeomorphisms]{Periodic symplectic and Hamiltonian diffeomorphisms on irrational ruled surfaces}
\author{Nicholas Lindsay}
\address{Department of Mathematics Education\\ Sungkyunkwan University \\ Seoul, Republic of Korea}
\email{ 20260280@skku.edu}

\author{Weiyi Zhang}
\address{Mathematics Institute\\  University of Warwick\\ Coventry, CV4 7AL, United Kingdom}
\email{Weiyi.Zhang@warwick.ac.uk}

\begin{abstract} 
We investigate when finite-order Hamiltonian diffeomorphisms extend to Hamiltonian circle actions, probing the transition from discrete to continuous symmetry in symplectic topology. Focusing on irrational ruled symplectic $4$-manifolds, we show that homologically trivial symplectic cyclic actions of order $k>2$ always extend to Hamiltonian $S^1$-actions, possibly after modifying the symplectic form. In contrast, we construct explicit symplectic involutions that cannot be so extended, even on minimal irrational ruled surfaces. These examples reveal geometric obstructions to extending discrete symmetries and highlight new exotic symplectic actions not equivalent to holomorphic ones. Our results also apply to higher-dimensional and non-cyclic group actions, and we establish several structural results on the isomorphism types of finite groups that can act on irrational ruled symplectic $4$-manifolds.



 \end{abstract}

\maketitle

\tableofcontents

\section{Introduction}
A fundamental theme across geometry, dynamics, and physics is the tension and interplay between discrete and continuous symmetries. This article is partially motivated by the following question of Kedra which connects the discrete and continuous symmetries in symplectic topology. More precisely, we consider:


\begin{question}\label{kedra}
For  a closed symplectic manifold $(M,\omega)$, does every Hamiltonian cyclic action extend to a Hamiltonian circle action?  
\end{question}

Here a Hamiltonian cyclic action refers to one generated by a periodic Hamiltonian diffeomorphism. Closed symplectic manifolds with a non-trivial Hamiltonian circle actions are symplectically uniruled \cite{Mc}. Moreover, in dimension $4$ uniruled symplectic manifolds are ruled\footnote{We follow the current convention that a ruled surface $S$ is a smooth projective surface birational to $C \times \mathbb{P}^1$, and it is geometrically ruled if additionally $b_{2}(S) = 2$. Note that in the terminology of \cite{Ma}, a ruled surface means geometrically ruled in our notation. A symplectic $4$-manifold is called (geometrically) ruled if it is diffeomorphic to a (geometrically) ruled algebraic surface. Note that an irrational ruled surface is minimal if and only if it is geometrically ruled.}.

Recent work, particularly in dimension $4$, has shed light on this question ({\it e.g.} \cite{CKr, CK, LLW}). Our approach also focuses initially on dimension $4$, guided by two key  viewpoints. First,  symplectic $4$-manifolds naturally divide into three classes: rational, irrational ruled, and non-uniruled, each requiring different treatment and yielding distinct answers.  
 Second, the behavior of involutions (order $2$) differs from high-order cyclic actions and should be treated separated.

For rational $4$-manifolds, {\it i.e.} symplectic manifolds which are diffeomorphic to $\mathbb {CP}^2\# k\overline{\mathbb {CP}^2}$ or $S^2\times S^2$,  the answer to Question \ref{kedra} is affirmative when $M=\mathbb {CP}^2$ \cite{Chen} or when $M$ is a rational geometrically ruled surface \cite{CKr} (that is, diffeomorphic to an $S^2$-bundle over $S^2$). However, Chiang and Kessler \cite{CK} constructed a homologically  trivial symplectic involution on a six-point blowup of $\mathbb {CP}^2$  that does not extend to a Hamiltonian circle action. This example was later shown to be Hamiltonian in \cite[Corollary 1.8]{LLW} and \cite[Remark 7]{AS}. 

No known examples exist of non-extendable symplectic $\mathbb Z_k$-actions with $k>2$ in the  rational case. In contrast, for non-uniruled symplectic manifolds,  Question \ref{kedra} aligns with a conjecture asserting that such manifolds admit no Hamiltonian torsion \cite{AS}. Progress on this front includes positive results  for symplectically aspherical manifolds \cite{Pol}, K3 surfaces \cite{CKw}, and symplectic Calabi-Yau and negative monotone symplectic manifolds of higher dimensions \cite{AS}. For more background, see Section \ref{back}. 

In dimension four, we focus on irrational ruled symplectic $4$-manifolds, {\it i.e.} $(M,\omega)$ symplectic blowups of $S^2$-bundles over  Riemann surfaces $\Sigma$ of genus at least one. For some results, we assume the base genus is at least $2$.  Since Hamiltonian diffeomorphisms are homologically trivial, we consider an alternative version of Question \ref{kedra}.

\begin{questionp}{\ref{kedra}$'$}\label{kedra'}
For an irrational ruled symplectic $4$-manifold $(M,\omega)$ with base genus at least two, does every homologically trivial symplectic cyclic action extend to a Hamiltonian circle action? 
\end{questionp}

The genus condition is essential. When the base genus is one, a simple example of homologically trivial symplectic cyclic action arises by taking the product of a finite-order rotation on the first factor of $T^2=S^1\times S^1$ with the trivial action on $S^2$. This action extends to a symplectic but not Hamiltonian circle action, even after modifying the symplectic form. In contrast, when the base genus is at least two, \cite{CK}  constructs a homologically trivial symplectic involution on a non-minimal irrational ruled symplectic $4$-manifold with $b_2(M)=5$, {\it i.e.}  an $S^2$-bundle over  $\Sigma_g$ blown up at three points, which does not extend to a Hamiltonian circle action.

Our first main result strengthens this by showing that such  involutions also exist on minimal irrational ruled surfaces, {\it i.e.}  irrational ruled symplectic $4$-manifolds of $b_2(M)=2$. This contrasts with the positive result of Chiang-Kessler for rational ruled $4$-manifolds. In fact, we establish the following stronger statement (see Corollary \ref{invtocircle}).

\begin{theorem}\label{main1}
There exist homologically trivial symplectic involutions on irrational ruled $4$-manifolds with base genus at least two and any $b_2\ge 2$, such that they cannot be extended to symplectic circle actions, even after modifying the symplectic form. 
\end{theorem}

These involutions are Hamiltonian diffeomorphisms assuming a positive answer to a question of Donaldson (Question \ref{don}) in the case of minimal ruled $4$-manifolds, and we discuss this further in Section \ref{back}. 

Our construction introduces a novel obstruction to cyclic-to-circle extension by analyzing the fixed point set. A homologically trivial symplectic circle action on a minimal irrational ruled surface with base genus at least two has a fixed point set consisting of two disjoint sections. This  also holds for  homologically trivial symplectic  $\mathbb Z_k$-actions with $k>2$ (Proposition \ref{sections}). However, we construct homologically trivial symplectic involutions with connected fixed point sets, a phenomenon that distinguishes them sharply from their higher-order cyclic counterparts (see Theorem \ref{exampleir}).

We further explores algebro-geometric realization of our construction in Section \ref{agreal} using Maruyama's classification of automorphism groups of ruled surfaces \cite{Ma}.

By comparing Theorem \ref{exampleir} with Lemma \ref{holfix}, our construction answers the following question posed by Weimin Chen \cite[Question 9]{chenq}: 

\begin{question}\label{chenq9}
Does there exist a symplectic finite group action on a K\"ahler surface which is not equivalent to a holomorphic action?
\end{question}
We affirmatively answer this as follows:

\begin{cor}\label{q9chen}
There exists a symplectic involution on a K\"ahler surface $(C, \omega_C) \times (\mathbb P^1, \omega_{\mathrm{FS}})$, where $(C, \omega_C)$ is a projective curve of genus at least one, which is not smoothly equivalent to any holomorphic action. 
\end{cor}
\begin{proof}
 For irrational ruled $4$-manifolds, cohomologous symplectic forms are symplectomorphic \cite{Mciso, LL}. In particular, any symplectic form on $C\times \mathbb P^1$ is symplectomorphic to a product form. Therefore, we may consider the product  $(C, \omega_C) \times (\mathbb P^1, \omega_{\mathrm{FS}})$, where $\omega_C$ is an area form on $C$ with total area $K\gg 0$. By Theorem \ref{exampleir}, this product admits a symplectic involution whose  fixed point set is connected. However, by Lemma  \ref{holfix}, any holomorphic action on the K\"ahler surface $C \times \mathbb P^1$ must have a disconnected fixed point set. It follows that the symplectic involution constructed in Theorem \ref{exampleir} cannot be smoothly equivalent to any holomorphic action. 
\end{proof}

One advantage of our construction is that we can extend our results to higher dimensions, in the spirit of  \cite[Question 1.10]{LLW}. The proof of non-extendability of Chiang-Kessler's involution relies on Karshon's description of $S^1$-Hamiltonian $4$-manifolds \cite{Kar}, which is not available in higher dimensions. The following corollary gives a sample how the higher dimensional examples could be constructed and how the non-extendability could be proved. The fixed point set of our involutions is non-empty as any finite order fixed point free symplectic action, such as the rotations on tori, automatically cannot extend to Hamiltonian circle actions.

\begin{cor}\label{main1'}
There exist homologically trivial symplectic involutions with non-empty fixed point set on $\Sigma_g\times S^2\times \Sigma_h$ where $g\ge 2$ and $h\ge 0$, which cannot be extended to a  Hamiltonian circle action, even with a possibly different symplectic form $\omega'$. 
\end{cor}
\begin{proof}
Consider the symplectic structure $\omega_M$ on $M=\Sigma_g\times S^2$, and let $R_M$ be a homologically trivial symplectic involution on $M$ as constructed in Theorem \ref{main1}. Define the product symplectic form $(\omega_M, \omega_{h})$ on $M\times \Sigma_h$, and consider the involution $R_M\oplus \mathrm{id}$ on this product. The fixed point set of the involution  is $S_0\times \Sigma_h$ where $S_0\subset M$ is the connected fixed point set of $R_M$, which is a surface of genus $2g-1$. 

Suppose first that $h=0$. In this case, the fixed point set lies in $S_0\times S^2$. For any Hamiltonian circle action on a closed symplectic manifold, the fixed point set must have at least two connected components corresponding to the maximal and minimal levels of the moment map. Each of these components must be a symplectic submanifold with fundamental group $\pi_1(\Sigma_g\times S^2\times S^2)=\pi_1(\Sigma_g)$ \cite{Lh}. But on the ruled surface $\Sigma_{2g-1}\times S^2$, any $2$-dimensional symplectic submanifold must either have genus zero or genus at least $2g-1>g$. Thus, $\Sigma_g$ cannot be realized as a symplectic submanifold of $S_0\times S^2$, and the involution $R_M\oplus \mathrm{id}$ cannot be extended to a  Hamiltonian circle action when $h=0$. 

Now consider $h>0$, if the involution were extendable, the fixed point set $\Gamma$ of the corresponding Hamiltonian circle action would be a symplectic submanifold of $S_0\times \Sigma_h$. However, since $S_0\times \Sigma_h$ is connected while $\Gamma$ must be disconnected  (again corresponding to the extremal values of the moment map), $\Gamma$ cannot contain a $4$-dimensional component. On the other hand, both the maximal and minimal level sets must have fundamental group $\pi_1(\Sigma_g\times \Sigma_h)$, which is not a surface group for $h>0$. This contradiction implies that $R_M\oplus \mathrm{id}$ is not extendable to a  Hamiltonian circle action in this case either. 
\end{proof}

A more general construction of potential higher dimensional examples,  based on the same idea as Theorem \ref{main1}, is presented in Proposition \ref{highex}, with explicit $6$-dimensional examples given in Corollary \ref{6dexample}.

In contrast to the case of symplectic involutions,  symplectic $\mathbb Z_k$-actions with $k>2$ exhibit significantly different behavior. 
The main result, Theorem \ref{theorem}, concerns 
 homologically trivial symplectomorphisms:

\begin{theorem} \label{main2}
Let $(M,\omega)$ be an irrational ruled symplectic $4$-manifold with $b_{2}(M)=2$, and suppose the base has genus $g 
\geq 2$. Let $f : M \rightarrow M$ be a homologically trivial symplectomorphism generating an effective $\mathbb{Z}_k$-action with $k>2$. Then  the $\mathbb{Z}_{k}$-action extends to a Hamiltonian $S^1$-action for some possibly different symplectic form $\omega'$, under the standard embedding $\mathbb{Z}_k \subset S^1$.
 \end{theorem}

In contrast, it is explained in \cite{CK} that there exist homologically trivial symplectic $\mathbb Z_{2k}$-actions that do not extend to a Hamiltonian circle actions on irrational ruled symplectic $4$-manifolds with $b_2(M)=2k+1$ for $k>1$.

A key ingredient in the proof of  Theorem \ref{main2} is  a recent advance by the second author \cite{Zh} on the existence and structure of $J$-holomorphic curves in irrational ruled symplectic $4$-manifolds, even with respect to arbitrary tamed almost complex structure. In particular,  Proposition \ref{bundle} shows that any finite-order symplectomorphism acting trivially on $H_{1}(M,\mathbb{Z})$ preserves the fibers of a fibration of $(M,\omega)$ by $J$-holomorphic spheres given by \cite[Theorem 1.2]{Zh}. For symplectomorphisms of order greater than two, we further derive detailed structural information about  the fixed point sets.


Throughout the article, we build on our results under progressively stronger homological assumptions on the symplectomorphisms, allowing for increasingly refined conclusions about their fixed point sets and invariant submanifolds.  For instance, we classify finite subgroups of symplectomorphisms of irrational ruled symplectic $4$-manifolds {\it en route} (Section \ref{finclass}), and we investigate  the discrete to continuous extension for  groups in this list other than the cyclic groups. In particular, we extend Theorem \ref{main1} to the Klein $4$-group (Theorem \ref{exoticklein}):

\begin{theorem}\label{main3}

There exists a homologically trivial symplectic action of the Klein $4$-group $\mathbb Z_2\times \mathbb Z_2$ on irrational ruled $4$-manifolds diffeomorphic to $\Sigma_g\times S^2$,  for any $g\ge 1$, such that two of its three non-trivial involutions do not extend to symplectic circle actions, even after modifying the symplectic form. 

Moreover, this Klein $4$-group action does not extend to a symplectic $\mathrm{SO}(3)$-action under any embedding $\mathbb Z_2\times \mathbb Z_2\subset \mathrm{SO}(3)$.  
\end{theorem}

In section \ref{agreal}, we show in Proposition \ref{threebisection} that there are algebraic actions of $\mathbb{Z}_2 \times \mathbb{Z}_2$ on certain geometrically ruled surfaces over an elliptic curve such that the fixed point set of every non-trivial element is an irreducible bisection.  Therefore the actions of any of the non-trivial elements do not extend to a Hamiltonian $S^1$-action.

With regard to the study of finite groups $G$ acting by symplectomorphisms on irrational ruled symplectic $4$-manifolds in Section \ref{finclass}, we greatly elucidate their structure by analyzing how they act on the fibration by $J$-holomorphic spheres constructed in  \cite[Theorem 1.2]{Zh}, c.f. Proposition \ref{bundle}. For example, when the base genus is $1$, we obtain a quite complete description of the possible isomorphism types of finite groups that can act by symplectomorphisms on an irrational ruled $4$-manifolds. The following result is proved in Section \ref{finclass}, it follows from Corollary \ref{finsym} and Corollary \ref{finsymone}. 

Before stating it, recall that any finite subgroup $H \subset \mathrm{SO}(3)$ is isomorphic to either a cyclic group $\mathbb{Z}_{p}$, a dihedral group $D_{p}$, or the symmetry group of a regular tetrahedron, octahedron, or icosahedron.  

\begin{theorem}\label{intro:finclass}
Let $(M,\omega)$ be an irrational ruled symplectic $4$-manifold with base surface $\Sigma$, and let $G$ be a finite group acting effectively by symplectomorphisms on $M$. \begin{itemize}
\item If $g(\Sigma)>1$, then there exists a short exact sequence $$1 \rightarrow H_{1} \rightarrow G \rightarrow H_{2} \rightarrow 1,$$ where $H_{1}$ is a finite subgroup of $\mathrm{SO}(3)$ and $H_{2}$ is a finite subgroup of $\mathrm{SL}_{2g}(\mathbb{Z})$.

\item If $g(\Sigma)=1$, then there exists a pair of short exact sequences $$1 \rightarrow H \rightarrow G \rightarrow H' \rightarrow 1 , \;\;\;\;1 \rightarrow H_1 \rightarrow H \rightarrow H_{2} \rightarrow 1,$$
such that:
\begin{enumerate}
\item $H'\cong \mathbb{Z}_{k}$ for $k = 1,2,3,4,6$;

\item  $H_1$ is isomorphic to a finite subgroup of $\mathrm{SO}(3)$; 

\item  There exist non-negative integers $p,q$ such that $$H_{2} \cong \mathbb{Z}_{p} \times \mathbb{Z}_{q}.$$
\end{enumerate}

\end{itemize}
\end{theorem}

We remark that our classification is based on analyzing the induced action on the first homology group. The induced action on second homology group cannot be directly detected from the triple $\{H', H_1, H_2\}$ or the short exact sequences, although it is classified separately in Proposition \ref{acth2}. In \cite{Mu}, it was shown that finite subgroups of the symplectomorphism group of $T^2 \times S^2$ with arbitrary symplectic form have the Jordan property. For some other work on finite subgroups of symplectomorphisms on irrational, geometrically ruled surfaces see \cite{Ca}.

The understanding of the finite subgroups of symplectomorphisms for other $4$-manifolds remains incomplete, even for rational $4$-manifolds, particularly $\mathbb {CP}^2$; see \cite{CLW} for partial results and  references. 

\medskip

\textbf{Acknowledgements} This project was initiated when N.L. visited the W.Z. in Summer 2019, N.L. would like to thank Shanghaitech University for funding the trip. The  authors would like to thank Richard Hind and Dmitri Panov for helpful discussions. We thank the referees for helpful comments. The authors were partially supported by the Warwick Horizon Europe Seed Fund.

\section{Preliminaries} \label{prel}
\subsection{Homologically trivial cyclic group actions on closed orientable surfaces of positive genus}
We begin with background results on periodic surface diffeomorphisms, derived from the Nielsen-Thurston classification of surface homeomorphisms \cite{T}. We state a special case sufficient for our purposes (see \cite{T} for the full classification).

\begin{theorem} \label{periodic} \cite{T}
Let $\Sigma$ be a closed surface of genus at least two. If $h : \Sigma \rightarrow \Sigma$ is a periodic homeomorphism that induces the identity on $H_{1}(\Sigma,\mathbb{Z})$, then $h = \mathrm{Id}_{\Sigma}. $\end{theorem}
\begin{proof}
Let $\Sigma$ be a closed orientable surface of genus at least two, and $h$ a finite-order homeomorphism acting trivially on $H_{1}(\Sigma,\mathbb{Z})$. First, we rule out the possibility that $h$ is orientation-reversing. Suppose otherwise. Then by Poincar\'{e} duality and the universal coefficients theorem, $$h^{*}: H^{1}(\Sigma,\mathbb{Z}) \rightarrow H^{1}(\Sigma,\mathbb{Z})$$ is the identity map. Choose $\alpha,\beta \in H^{1}(\Sigma,\mathbb{Z}) $ such that $\alpha\cdot \beta \neq 0$ ({\it e.g.}, the Poincar\'{e} dual of two simple closed curves intersecting at a single point). Then $h^{*}(\alpha\cdot \beta) = \alpha\cdot \beta$,  contradicting the fact that $h^{*}$ must act as multiplication by $-1$ on $H^{2}(\Sigma,\mathbb{Z})$. Thus, $h$ is orientation preserving.

 By \cite[Theorem 6.8]{FM}, $h$ is isotopic to the identity. Therefore by Nielsen-Thurston classification of surface homeomorphisms \cite{T}, $h=\mathrm{Id}_{\Sigma}$ (for a concise reference for the statement needed see \cite[Lemma 4.4]{JS}).
\end{proof}

Next, we consider the case where $\Sigma=T^2$. In this setting, there exist non-trivial examples of cyclic group actions that acts trivially on homology, as shown in Example \ref{torusexample}. For convenience, we realize $T^2$ as the Clifford torus $$T^2=\{(z_1,z_2) \in \mathbb{C}^2 | \, |z_1| = |z_2|=1\}.$$

\begin{example}\label{torusexample} Let $p,q$ be  coprime positive integers. Define the diffeomorphism $f_{p,q}$ by $$f_{p,q}(z_1,z_2) = (e^{\frac{ 2\pi iq}{p}}z_1,z_2). $$ This map generates an action of $\mathbb{Z}_{p}$ which acts trivially on homology. \end{example}

While there are non-trivial examples in the genus one case,  cyclic homologically trivial actions on $T^2$ admit a complete classification. Before stating this, we recall a  simpler fact, which follows directly from the Lefschetz fixed point formula. 

Recall that a group action is called free if the only group element acting as the identity map is the identity element itself.

\begin{lemma} \label{freelemma}
An effective action of $\mathbb{Z}_{p}$ on $T^2$ by homeomorphisms that acts trivially on $H_{1}(T^2,\mathbb{Z})$ is free.
\end{lemma} 

The lemma has the following consequence:
\begin{lemma} \label{torusgroups}
Let $G$ be a finite group acting effectively by homeomorphisms on $T^2$ and trivially on  $H_{1}(T^2,\mathbb{Z})$. There there exist non-negative integers $p,q$ such that $$G \cong \mathbb{Z}_{p} \times \mathbb{Z}_q.$$
\end{lemma}
\begin{proof}
By Lemma \ref{freelemma} the action of $G$ is free. Since it is orientation-preserving, the quotient map $\pi: T^2 \rightarrow T^2/G$ is a covering, with $T^2/G \cong T^2$. Therefore by \cite[Proposition 1.39]{H}, $$G \cong \pi_{1}(T^2,\mathbb{Z})/ \pi_{*}(\pi_{1}(T^2,\mathbb{Z})),$$ and considering the possible quotients of $\mathbb{Z}^2$ by subgroups isomorphic to $\mathbb{Z}^2$ gives the result.
\end{proof}

A classical theorem of Nielsen and Smith further asserts that cyclic actions on surfaces are classified, up to conjugacy in the homeomorphism group, by the induced map on integral homology of the quotient $$\pi_{*} : H_{1}(\Sigma) \rightarrow H_{1}(\Sigma/\mathbb{Z}_{p}).$$ In the case of $T^2$, both the surface $\Sigma$ and its quotient $\Sigma/\mathbb{Z}_{p}$ are homeomorphic to $T^2$. 

For the actions in Example \ref{torusexample},  the image of $\pi_*$ depends only on $p$: it is the subgroup generated by $(p,0)$ and $(0,1)$. However the induced homomorphism also depends on $q$. 

To state the classification precisely, we recall the following definition.
\begin{definition}
Two actions of $\mathbb{Z}_{p}$ on $T^2$, generated by homeomorphisms $f_1,f_2$ of order $p$, are said to be topologically conjugate if there exists a homeomorphism $h$ such that $$f_{2} = h f_1 h^{-1}.$$
\end{definition}

A result of Baranov, Grines, Pochinka and Chilina \cite[Theorem 2]{BGVC}, obtained via Nielsen theory, establishes that the actions in Example \ref{torusexample} account for all possible homologically trivial  $\mathbb{Z}_{p}$-actions on $T^2$, up to topological conjugacy. 

\begin{theorem} \label{shift} \cite[Theorem 2]{BGVC} 
Any action of $\mathbb{Z}_p$ on $T^2$ by homeomorphisms that acts trivially on $H_{1}(T^2,\mathbb{Z})$ is topologically conjugate to one of the actions from Example \ref{torusexample}.
\end{theorem}

\subsection{Almost complex structures}
We now collect the necessary tools concerning almost complex structures and  pseudoholomorphic curves. 
A key input is the following theorem from \cite{Zh}, establishing the existence and behavior of $J$-holomorphic curves on irrational ruled symplectic $4$-manifolds. 

For an arbitrary tamed almost complex structure, the existence of a pseudoholomorphic curve in the fiber class passing through every point of $M$ follows from \cite[Theorem 1.2]{Zh} by applying the $J$-nef technique developed in \cite{LZ}. The continuity of the resulting map to the moduli space of fibers $\Sigma$ is established in \cite[Corollary 3.9]{Zh}. Moreover,  the fact that $M$ admits the structure of a smooth fiber bundle in a neighborhood of an irreducible fiber is explained immediately below the proof of that corollary.

\begin{theorem} \cite{Zh}  \label{fibration} Let $(M,\omega)$ be an irrational ruled symplectic $4$-manifold with base genus $g \geq 1$, and let $F \in H_{2}(M,\mathbb{Z})$ denote the fiber class. 
 Then for any $\omega$-tamed almost complex structure $J$, there exists a continuous map $f: M \rightarrow \Sigma$, where $\Sigma$ is a surface of genus $g$, such that: 
\begin{enumerate}
\item[a)] For all $p \in \Sigma$, the fiber $f^{-1}(p)$ is a connected $J$-holomorphic curve in the class $F$, whose irreducible components are spheres with non-positive self-intersection. 

\item[b)] For all but finitely many $p \in \Sigma$, $f^{-1}(p)$ is a smooth $J$-holomorphic sphere with self-intersection $0$. On a neighborhood of $f^{-1}(p)$, $f$ defines a smooth fiber bundle.

\item[c)] If $b_{2}(M)=2$, then $f$ is a smooth $S^2$-bundle whose fibers are $J$-holomorphic spheres representing the class $F$.
\end{enumerate}

\end{theorem}

We refer to the map $f$ in Theorem \ref{fibration} as the moduli map:

\begin{definition}
Let $(M,\omega)$ an irrational ruled symplectic $4$-manifold, and $J$ a tamed almost complex structure. Then, the continuous map $f: M \rightarrow \Sigma$ provided by Theorem \ref{fibration} is called the moduli map. 
\end{definition}

We will analyze the behavior of this fibration under cyclic group actions. The following lemma, used repeatedly in our arguments, shows that the moduli map induces an isomorphism on first integral homology group. For a topological space $X$, let $\tilde{H}_i(X)$ denote the reduced integral homology of $X$.

\begin{lemma} \label{modfirsthom}
Let $(M,\omega)$ be an irrational ruled symplectic $4$-manifold, and $J$ a tamed almost complex structure. Then the moduli map $f : M \rightarrow \Sigma$ induces an isomorphism $f_{*}: H_{1}(M,\mathbb{Z}) \rightarrow H_{1}(\Sigma,\mathbb{Z}) $. 
\end{lemma}
\begin{proof}
From Theorem \ref{fibration}, for each $p \in \Sigma$, the fiber $f^{-1}(p)$ is a connected union of smooth $J$-holomorphic rational curves. By \cite[Theorem 1.5]{LZ}, the configuration of components \footnote{ The configuration of curves in this context is (as an unlabeled graph) the graph whose vertices are irreducible components of $f^{-1}(p)$ and edges correspond to intersection points \cite[Definition 3.4]{LZ}.} forms a tree. Hence the reduced homology $\tilde{H}_{i}(f^{-1}(p)) = 0$ for $i=0, 1$. Then by the Vietoris-Begle mapping theorem (see \cite{VN} or \cite[page 95]{D}), $f_{*}: H_{1}(M,\mathbb{Z}) \rightarrow H_{1}(\Sigma,\mathbb{Z})$ is an isomorphism. 
\end{proof}

Another important tool  is the evaluation map associated with moduli spaces of pseudoholomorphic spheres, which will be used in the proof of Theorem \ref{theorem}. We recall the definition below:

\begin{definition} \label{evmap}
Let $(M,\omega)$ be a closed symplectic manifold, $A \in H_{2}(M,\mathbb{R})$, and $J$ an almost complex structure. Then: \begin{itemize}
\item[a)] Let $\mathcal{M}(A,J)$ denote the set of non-multiply covered $J$-holomorphic maps $u: S^2 \rightarrow M$ representing the class $A$, {\it i.e.}, $u_{*}[S^2] = A$.
\item[b)]  The moduli space of unparameterized curves is the quotient $$\mathcal{M}(A,J)/G, $$ where $G = \mathrm{PSL}(2,\mathbb{R})$ is the space of M\"{o}bius transformations acting on  $\mathcal{M}(A,J)$ by reparametrization.
\item[c)] The evaluation map is defined as: $$ev:  \mathcal{M}(A,J) \times_{G} S^2 \rightarrow M,$$ where $G$ acts on $\mathcal{M}(A,J) \times S^2 $ by the diagonal action $g\cdot (u,z) = ( u \circ g, g^{-1}(z))$. 
\end{itemize}
\end{definition}

The following useful lemma ensures the compatibility of almost complex structures with group actions, which is an almost complex version of \cite[Theorem 4.3.1]{P}:

\begin{lemma} \label{inv}
Suppose that a finite group $G$ acts effectively by symplectomorphisms on a closed symplectic manifold $(M,\omega)$. Then there exists a $G$-invariant, $\omega$-compatible almost complex structure $J$. Moreover, the fixed point set $\mathrm{Fix}(G)$ is an almost complex submanifold of $M$.
\end{lemma}
\begin{proof}
Let $g_0$ be a Riemannian metric compatible with $\omega$. Averaging over the group $G$, we define: $$g:=\frac{1}{|G|}\sum_{h\in G}h^{*} g_0,$$ which yields a $G$-invariant Riemannian metric. Since each $h\in G$ is a symplectomorphism, the averaged metric $g$ remains compatible with $\omega$. 

We now apply the map constructed in \cite[Proposition 2.5.6 (i)]{MS} (or Proposition 2.50 in the second edition of \cite{MS}) to this invariant metric. The resulting almost complex structure is $G$-invariant and compatible with $\omega$.  

We now explain why this follows from  \cite[Proposition 2.5.6]{MS} and sketch the definition of the map $r$.  Let $(V,\omega)$ be a symplectic vector space, and $\mathcal{M}et(V)$ denote the set of metrics on $V$ and $\mathcal{J}(V,\omega)$ of $\omega$-compatible almost complex structures. Recall that Sp$(V)=\hbox{Sp}(V, \omega)$ is the set of linear automorphisms of $V$ preserving the symplectic structure.

For any symplectic vector space $(V,\omega)$, there is a map $$r: \mathcal{M}et(V) \rightarrow \mathcal{J}(V,\omega)$$ assigning to each Riemannian metric on $V$ an $\omega$-compatible almost complex structure. The key properties of $r$ are: 
\begin{itemize}
\item $r(g_{J}) = J$ for any $J \in \mathcal{J}(V,\omega)$,
\item  $r (\Phi^{*}(g)) = \Phi^{*} (r(g)),$ for all $\Phi \in \hbox{Sp}(V)$, $g \in \mathcal{M}et(V)$.
\end{itemize}

Applying this to the $G$-invariant metric $g$, and using the fact that each group element acts symplectically, we find $$  a^{*}(r(g)) = r(a^{*}g)  = r(g)$$  for any $a \in G$. Hence, the resulting $\omega$-compatible almost complex structure $J:=r(g)$ is $G$-invariant.

To briefly recall the construction of $r$:  for $g \in  \mathcal{M}et(V)$,  there exists a linear automorphism $A: V\rightarrow V$ such that $$\omega(Au,v) = g(u,v) \quad \hbox{for all } u, v\in V.$$ Let $P = A^{*}A$.  Then there exists a  linear automorphism $Q$ such that $Q^{2}=P$.  The compatible almost complex structure is then given by  $$r(g)=Q^{-1}A.$$  Full details are provided in 
\cite[Proposition 2.5.6]{MS}

For the final statement, we know that the fixed point set $\mathrm{Fix}(G)$ of a smooth action by a finite group $G$ on a closed manifold is a disjoint union of smooth connected submanifolds. Let $x\in \mathrm{Fix}(G)$.  Since $J$ is $G$-invariant, the tangent space $T_xM$ inherits a complex structure $J_x$, and the fixed subspace $T_x(\mathrm{Fix}(G))\subset T_xM$ is preserved by $J_x$. Therefore, $\mathrm{Fix}(G)$ is an almost complex submanifold of $M$.
\end{proof}

\subsection{Weights of cyclic group actions} \label{weightsprel}
In this section, we recall the standard notion of weights for cyclic group actions on almost complex manifolds. 

Suppose the finite cyclic group $\mathbb{Z}_{k}$ acts on an almost complex manifold $(M,J)$, preserving the almost complex structure $J$. If $p\in M$ is a fixed point of the action,  then the tangent space $T_{p}M$ inherits the structure of a complex $\mathbb{Z}_k$-representation. 

Recall that the $k$ irreducible complex representations of $\mathbb{Z}_k$ are indexed by integers $w \mod k$, with generator $\alpha = e^{\frac{2 \pi i w}{k}}$. For a given $m \in \mathbb{Z}/k \mathbb{Z}$, the group action is given by $$m. z = \alpha^{m}.z.$$

Let $F\subset M$ be a fixed point component of the $\mathbb Z_k$-action. By Lemma \ref{inv}, $F$ is an almost complex submanifold. The only fact about weights we will need in this article is the following: For any $p \in F$, the $\mathbb{Z}_k$-representation on $T_{p}(M)$ is  independent of the choice of $p $. In particular, the weights encoding the decomposition of  $T_{p}(M)$ are constant along $F$. The proof is elementary and left to the reader.

\subsection{Automorphism groups of irrational ruled surfaces} \label{aut}

Let $S$ be a smooth projective surface over $\mathbb C$. The surface $S$ is called geometrically ruled if there exists a morphism $\pi : S \rightarrow C$, where $C$ is a smooth curve, such that for every point $p \in C$, $\pi^{-1}(p) \cong \mathbb{P}^1$. Note that $S$ is irrational if and only if the base curve $C$ is irrational. It is known that every geometrically ruled surface is isomorphic to the projectivization $\mathbb{P}(E)$ of a rank two vector bundle   $E$ over $C$ \cite{Ma}.

In this section, we  recall several results concerning the automorphism groups of such irrational ruled algebraic surfaces. The automorphism groups of geometrically ruled surfaces were classified by Maruyama in \cite{Ma}. We summarize below the aspects of this classification most relevant to our setting, particularly for irrational geometrically ruled surfaces. We recall the following important definition from \cite{Ma}:

\begin{definition} \cite{Ma} \label{Mac}
Let $(S ,\pi,C)$ be a geometrically ruled surface, so that $S \cong \mathbb{P}(E)$, where $E$ is a rank two vector bundle over $C$. Then: \begin{itemize}
\item[a)] Denote by $\mathrm{Aut}_{C}(S) $ the subgroup of $\mathrm{Aut}(S)$ given by $$\mathrm{Aut}_{C}(S) = \{g \in \mathrm{Aut}(S)| \pi \circ g = \pi\}.$$
\item[b)] A line subbundle of $E$ which has the maximal degree is called a maximal subbundle of $E$, and $M(E)$ denote the degree.
\item[c)] Set $N(E) = \deg(E)-2M(E).$ 

\item[d)] Let $\Delta\subset \mathrm{Jac}(C)$ be the subgroup of the Jacobian defined by $$\Delta := \{L\in \mathrm{Jac}(C) | E \otimes L \cong E\}.$$
\end{itemize} 
\end{definition}
The subgroup $\mathrm{Aut}_{C}(S) $ is particularly important in the paper, as any homologically trivial automorphism of $S$ must lie in this subgroup when the base curve $C$ has genus at least two. It is known that the quantity $N(E)$ is an invariant of the surface $S$ and plays a key role in the classification of automorphism groups \cite[Theorem 2]{Ma}. The subgroup $\Delta\subset \mathrm{Jac}(C)$ defined in Definition \ref{Mac}(d) is a subgroup of the $2$-torsion part of the Jacobian variety of $C$, and is therefore finite. 

The following short exact sequence constructed in \cite{Gr} will be important in our analysis.

\begin{lemma} \label{exact}  Let $(S ,\pi,C)$ be a geometrically ruled surface with $S \cong \mathbb{P}(E)$. Then, there is a short exact sequence  $$1 \rightarrow \mathrm{Aut}(E)/\mathbb{C}^* \rightarrow \mathrm{Aut}_{C}(S) \rightarrow \Delta \rightarrow 1,$$ where $\mathbb{C}^*$ acts on $E$ by scaling.
\end{lemma}
\begin{proof}
This results is stated in  \cite[Proposition 2-Corollary, Page 202]{Gr} for vector bundles of arbitrary rank $n$. Let $E_{0}$ be a vector space of rank $n$. Then there is a exact sequence of groups $$1 \rightarrow \mathbb{C}^* \xrightarrow{i} \mathrm{Aut}(E_0) \xrightarrow{p} \mathrm{Aut}(\mathbb{P}(E_0)) \rightarrow 1,$$ where $i(\lambda) = \lambda \cdot \mathrm{Id}$ and $p(T) = [T]$ is the induced projective transformation. 

Passing to the corresponding locally free sheaves over the base curve $C$, one obtains a long exact sequence in cohomology: $$1 \rightarrow H^{0}(C, \mathcal{O}^*) \xrightarrow{\varphi_1} H^{0}(C, \mathrm{Aut}(E)) \xrightarrow{\varphi_2} \mathrm{Aut}_{C}(X)\rightarrow  H^{1}(C, \mathcal{O}^*) \rightarrow \ldots$$
Here, $ H^{0}(C, \mathcal{O}^*)\cong \mathbb{C}^*$, and  $ H^{1}(C, \mathcal{O}^*)\cong \mathrm{Jac}(C)$, the group of line bundles over $C$. In  \cite[Proposition 2-Corollary, Page 202]{Gr}, it is shown that the image of $\phi_2$ is exactly $\Delta$, and that each $L \in \Delta$ satisfies $L^{\otimes n} \cong \mathcal{O}$. Applying this in the case $n=2$ yields the claimed exact sequence. 
\end{proof}

Recall that a section of a geometrically ruled surface $\pi: S\rightarrow C$ is a morphism $f: C \rightarrow S$ such that $\pi \circ f = \mathrm{Id}_{C}$. A section is called minimal if it has the smallest possible self-intersection among all sections of $S$. The following lemma describes the relationship between minimal sections and the vector bundle structures: 

\begin{lemma} \label{intsec} \cite[Lemma 5]{Ma} Let $(S ,\pi,C)$ be a geometrically ruled surface, with $S \cong \mathbb{P}(E)$. Then:
\begin{itemize}
\item Minimal sections of $S$ correspond bijectively to maximal subbundles of $E$.
\item Moreover,  the self-intersection number of a minimal section  is equal to $N(E)$.
\end{itemize}
\end{lemma}

Finally, we highlight the special case where the geometrically ruled surface is a product $C\times \mathbb{P}^1$. This case is distinguished in \cite[Theorem 2]{Ma} and,  in a sense, admits a more restrictive automorphism group than a general projective bundle. Since the following is claimed without proof in \cite{Ma} we give a proof below for completeness.
\begin{lemma}  \label{one}Let $C$ be a smooth projective curve of genus $g\ge 1$, and let $S = C \times \mathbb{P}^1$. Then: $$\mathrm{Aut}(S) \cong \mathrm{Aut}(C) \times \mathrm{Aut}(\mathbb{P}^1). $$
\end{lemma}
\begin{proof}
Let $\pi_1: S \rightarrow C$ and $\pi_2: S \rightarrow \mathbb{P}^1$ denote the projections. Let $F_{C}$ denote the homology classes of $C \times \{p\}$, and $F_{\mathbb P^1}$ the  class of $\{q\} \times \mathbb{P}^1$, for  $p \in \mathbb P^1$, $q \in C$. Then the intersection pair satisfies: \begin{equation} \label{int} (a F_{C} + b F_{\mathbb P^1})^2 = 2ab. \end{equation}

Let $\psi : S \rightarrow S$ be an automorphism. Consider the image $C_1 = \psi(\{p\} \times \mathbb{P}^1)$ for a fixed $p \in C$. This is a smooth rational curve in $S=C \times \mathbb{P}^1$ with self-intersection zero. Since $C_1^2=0$, and from Equation \eqref{int}, $C_{1}$ must be of the form $\{p'\} \times \mathbb{P}^1 $ for some $p'\in C$. That is, $\pi_{1}|_{C_1}$ is constant.

Similarly, for any $q \in \mathbb{P}^1$, the image $C_2 = \psi(C \times \{q\})$ is a smooth curve of genus $g$ with self-intersection zero. By the above, $C_2\cdot F_{\mathbb P^1}=1$. Since $C_{2}^2 = 0$ and $C_2\cdot F_{\mathbb P^1}=1$, Equation  (\ref{int}) implies that $[C_2] = [F_C]$. In particular, $(\pi_{2})_{*}[C_2] = 0$, and so $\pi_{2}|_{C_2}$ is constant. That proves that $C_2 = C \times \{q'\}$ for some $q' \in \mathbb{P}^1$.

Thus $\psi$ preserves the product structure and must be of the form $\psi = \psi_{1} \times \psi_{2}$, where $\psi_{1} \in \mathrm{Aut}(C) $ and $\psi_{2} \in \mathrm{Aut}(\mathbb{P}^1)$.
\end{proof}

The following lemma follows as an immediate consequence.

\begin{lemma} \label{holfix}
 Let $C$ be a smooth projective curve of genus $g\ge 1$, and let $S = C \times \mathbb{P}^1$. Suppose $\psi \in \mathrm{Aut}(S)$ is a finite order automorphism of order $k$,  acting trivially on $H_{1}(S,\mathbb{Z})$. Then: $$\psi = \mathrm{Id}_{C} \times \tilde{\psi},$$ where $\tilde{\psi}$ is an elliptic M\"{o}bius transformation of order $k$. In particular, the fixed point set of $\psi$ consists of two disjoint sections.
\end{lemma}
\begin{proof}
By Lemma \ref{one}, we have $\psi = \psi_1 \times \psi_2$, where $\psi_{1} \in \mathrm{Aut}(C) $ and $\psi_{2} \in \mathrm{Aut}(\mathbb{P}^1)$. Since $\psi$ acts trivially on $H_1(C, \mathbb Z)$, and $H_1(S, \mathbb Z)\cong H_1(C, \mathbb Z)$, it follows that $\psi_{1}$ must act trivially on $H_1(C, \mathbb Z)$. But as $C$ has genus at least one, it follows from Theorem \ref{periodic} that $\psi_1=\mathrm{Id}_C$. Hence, $\psi=\mathrm{Id}_C\times \psi_2$. 

Since $\psi$ has finite order $k$, so does $\psi_{2}$, and thus $\psi_{2}\in \mathrm{Aut}(\mathbb P^1)$ is an elliptic M\"obius transformation of $\mathbb{P}^1$. Elliptic M\"obius transformations have exactly two fixed points on $\mathbb P^1$, so the fixed point set of $\psi$ is the union of two disjoint sections. 
\end{proof}

\subsection{Background on Question \ref{kedra} and related work}\label{back}
For surfaces, Question \ref{kedra} has an affirmative answer. The analysis naturally divides into two cases, depending on whether the surface $M$ is a sphere or has higher genus.
\bigskip

\noindent{{\bf Case 1: $M=S^2$}} 

\medskip
We present two arguments for this case. 

\medskip
\noindent{{\bf First Argument (Topological Conjugacy):}}
\medskip

A classical result of K\'er\'ekjart\`o and Eilenberg \cite{CKem} provides a topological classification of periodic homeomorphisms of the $2$-sphere:

\begin{theorem}\label{KE}
Let $f: S^2\rightarrow S^2$ be a periodic homeomorphism. Then there exists $r\in \mathrm{O}(3)$ and a homeomorphism $h:S^2\rightarrow S^2$ such that $$f=h\circ r\circ h^{-1}.$$
\end{theorem}

A few remarks follow. First, since Hamiltonian diffeomorphisms preserve area, we must have $r\in \mathrm{SO}(3)$. Second,  if $f$ is a diffeomorphism , then the conjugating map $h$ can also be taken to be a diffeomorphism (see, {\it e.g.}, \cite{GL}). 

Since every element of $\mathrm{SO}(3)$ is a rotation, it lies in a circle subgroup of $\mathrm{SO}(3)$. Thus, the diffeomorphism $f$, being conjugate to $r$, is also contained in a circle action. This circle action is clearly Hamiltonian and, up to conjugation, essentially unique. 

\medskip
\noindent{{\bf Second Argument (Holomorphic Realization):}}
\medskip

This approach, which forms the base case for later discussions of irrational ruled $4$-manifolds,  uses almost complex geometry. By Lemma \ref{inv}, any periodic symplectomorphism $f$ preserves some $\omega$-compatible almost complex structure $J$. On $S^2$, this $J$ is biholomorphic to the standard complex structure by the uniformization theorem. The group of biholomorphisms of $\mathbb {CP}^1$ is the M\"obius group, and every M\"obius transformation is conjugate to an element of SO$(3)$. Hence, $f$ again lies in a Hamiltonian circle action.  
\bigskip

\noindent{{\bf Case 2: $M$ is a surface of genus $g\ge 1$}} 

\medskip

In this case, we show that no non-trivial periodic Hamiltonian diffeomorphisms can exist. To this end, we apply \cite[Theorem 1.4]{FH}: 

\begin{theorem}\label{FHhyp}
Suppose $F: S\rightarrow S$ is a non-trivial Hamiltonian diffeomorphism of a closed, oriented surface $S$ of genus at least one. Then there exist integers $n>0$ and $p>0$ such that $F^n$ has a periodic point of period $k$ for every $k\ge p$. 
\end{theorem}

This theorem implies that any non-trivial Hamiltonian diffeomorphism on a higher genus surface cannot be of finite order, and thus  no non-trivial periodic Hamiltonian diffeomorphisms exist on such surfaces.

Thus, we obtain the following consequence: 

\begin{prop}\label{surfext}
Any periodic Hamiltonian diffeomorphism on a closed surface extends to a Hamiltonian circle action. 
\end{prop}

In higher dimensions, the analysis naturally divides into two cases, analogous to the surface case: uniruled and non-uniruled symplectic manifolds. For the non-uniruled manifolds, there are results analogous in spirit to Theorem \ref{FHhyp}, many of which are inspired by the Conley conjecture \cite{GGbey}.

In general, the following inclusions hold among groups of symplectomorphisms: $$\mathrm{Ham}(M, \omega)\subset \mathrm{Symp}_0(M, \omega)\subset \mathrm{Symp}_h(M, \omega),$$ where:
\begin{itemize}
\item Ham$(M, \omega)$ is the group of Hamiltonian diffeomorphisms,
\item Symp$_0(M, \omega)$ is the identity component of the  symplectomorphism group, and 
\item Symp$_h(M, \omega)$ denotes the group of symplectomorphisms acting trivially on homology.
 \end{itemize}
  Questions \ref{kedra} and \ref{kedra'} ask whether a finite subgroup of Ham$(M, \omega)$ or Symp$_h(M, \omega)$ can be extended to a circle subgroup of $\mathrm{Ham}(M, \omega)$. 

A natural variant of the question arises for Symp$_0(M, \omega)$. When $H^1(M, \mathbb R)=0$, Ham$(M, \omega)= \mathrm{Symp}_0(M, \omega)$. In many other situations,  finite subgroups of Symp$_0(M, \omega)$ are known to lie inside Ham$(M, \omega)$. This is the case, for example, for irrational ruled $4$-manifolds with base genus at least two. For more general results, see \cite{A}. 

The other inclusion $\mathrm{Symp}_0(M, \omega)\subset \mathrm{Symp}_h(M, \omega)$ is less well understood. In the rational $4$-manifold case, there has been solid progress in computing the symplectic Torelli group $$\mathrm{Symp}_h(M, \omega)/\mathrm{Symp}_0(M, \omega),$$ wee \cite{LLW} for effective computations. A central open question in this direction, attributed to Donaldson (see \cite{SS, LLW}), is the following: 

\begin{question}[Donaldson]\label{don}
Is the symplectic Torelli group generated by squared Dehn twists along Lagrangian spheres?
\end{question}

In the case of a minimal irrational ruled  $4$-manifold, a positive answer to Question \ref{don} would imply that the symplectic Torelli group is trivial, as such manifolds contain no Lagrangian spheres. Combining this with the results of \cite{A} mentioned above, one would conclude that the symplectic involutions constructed in Theorem \ref{exampleir} are in fact Hamiltonian diffeomorphisms when the base genus is at least two. \footnote{We note that a weaker version of Question \ref{don} was posed by McDuff and Salamon (\cite[Problem 14]{MS}): whether every symplectomorphism of a minimal ruled $4$-manifold  that is smoothly isotopic to the identity is also symplectically isotopic to the identity.  This question was affirmatively answered in \cite{Mcsacs} for the symplectic forms considered in Theorem \ref{exampleir}, provided the constant $K>g$. Consequently, to show that our homologically trivial symplectic involutions are Hamiltonian, it suffices to prove that they are smoothly isotopic to the identity.

However, the results of Shevchishin-Smirnov \cite{SSell} and Buse-Li \cite{BL} show that McDuff-Salamon's question fails for irrational ruled $4$-manifolds with $b_2=3$. In particular, this provides counterexamples to Donaldson's Question \ref{don} in that setting, since such manifolds contain no Lagrangian spheres, yet their symplectic Torelli group is non-trivial.}

For non-uniruled symplectic manifolds, it is  known that they do not admit Hamiltonian circle actions \cite{Mc}. In this setting, a positive answer to Question \ref{kedra} is equivalent to the statement that such manifolds admit no finite order Hamiltonian diffeomorphisms. In particular, if a closed symplectic manifold $(M, \omega)$ admits a Hamiltonian diffeomorphism of finite order, then there must exist a spherical homology class $A\in H_2(M, \mathbb Z)$ such that $$\langle c_1(TM), A\rangle>0 \hbox{ and } \langle [\omega], A\rangle>0.$$ In contrast, the standard rotational finite-order elements in Symp$_0(T^4, \omega_{\mathrm{std}})$  are clearly not Hamiltonian.  For the Enriques surface, we have Ham$=\mathrm{Symp}_0$, since $H^1(M, \mathbb R)=0$. However, it remains  unclear to the authors whether the homologically trivial involutions of Enriques surfaces classified in \cite{MN} are Hamiltonian.

\section{Finite symplectic actions on irrational ruled $4$-manifolds}
In this section, we apply pseudoholomorphic curve techniques to study finite  group actions by symplectomorphisms on irrational ruled symplectic $4$-manifolds.

\subsection{Construction of the fibration}
In this subsection, we provide a complete description of the isomorphism types of finite groups which can act symplectically on an irrational ruled symplectic $4$-manifold with base genus at least $2$, under the assumption that the action is trivial on the first integral homology group. In Corollaries \ref{finact} and \ref{finacttorus} of the next subsection, Proposition \ref{bundle} will be applied to further describe the possible isomorphism types of such groups.

\begin{prop} \label{bundle}
Let $(M,\omega)$ be an irrational ruled symplectic $4$-manifold with base surface $\Sigma$. Suppose a finite group $G$ acts effectively by symplectomorhisms on $M$ and trivially on $H_{1}(M,\mathbb{Z})$. Then:

\begin{itemize}
\item[a)] If $g(\Sigma) > 1$,  there exists a $G$-invariant, $\omega$-compatible almost complex structure $J$ such that the fibers of the associated moduli map $f: M \rightarrow \Sigma$ are preserved by the $G$-action. Moreover, the action of $G$ on each fiber is effective.

\item[b)]  If $g(\Sigma) = 1$, there exists a  $G$-invariant, $\omega$-compatible almost complex structure $J$ and a moduli map $f: M\rightarrow \Sigma$ such that, with  $$G_{1} := \{g \in G| f \circ g =f\},$$ we have a short exact sequence $$1 \rightarrow G_1 \rightarrow G \rightarrow G_{2} \rightarrow 1 ,$$ where:

\begin{itemize} \item[b1)] $G_{1}$ acts effectively on fibers of $f$;

\item[b2)]  $G_2$ acts effectively on $\Sigma$ by homeomorphisms and trivially on $H_{1}(\Sigma,\mathbb{Z})$; 

\item[b3)] There exist non-negative integers $p,q$ such that $G_{2} \cong \mathbb{Z}_{p} \times \mathbb{Z}_{q}.$

\end{itemize}

 \end{itemize}

\end{prop}
\begin{proof}
Firstly by Lemma \ref{inv} there is a compatible almost complex structure $J$ invariant by the action of $G$. Theorem \ref{fibration} provides a fibration  $f: M\rightarrow \Sigma$ by $J$-holomorphic spheres.   

Next, let $\varphi$ be the symplectomorphism corresponding to the action of any non-trivial element in $G$. Since $\pm[F]$ are the only homology classes with square zero that can be represented by embedded spheres, it follows that $\varphi_*[F]=\pm[F]$. However, as $\varphi$ is a symplectomorphism preserving orientation, we must have $\varphi_*[F]=[F]$. 

Moreover, $\varphi$ preserves the almost complex structure $J$, and since it preserves the fiber class $[F]$, it must map each fiber to another fiber. Therefore,  homeomorphism $\varphi: M\rightarrow M$ induces  a homeomorphism $\tilde{\varphi} : \Sigma \rightarrow \Sigma$ of the base surface $\Sigma$, fitting into the following commutative diagram: $$\begin{CD}
H_{1}(M, \mathbb Z)    @>\varphi_{*}>>  H_{1}(M, \mathbb Z)\\
@VVf_{*}V        @VVf_{*}V\\
H_{1}(\Sigma, \mathbb Z)     @>\tilde{\varphi}_*>>  H_{1}(\Sigma, \mathbb Z)
\end{CD} $$ 
By assumption, $\varphi_{*}=\mathrm{Id}$. Since $f_{*} : H_{1}(M,\mathbb{Z}) \rightarrow H_{1}(\Sigma,\mathbb{Z}) $ is an isomorphism by Lemma \ref{modfirsthom}, it follows that $\tilde{\varphi}_{*} =\mathrm{Id}$.

If $g(\Sigma)>1$, since $\tilde{\varphi}$ is a periodic homeomorphism acting trivially on first homology, Theorem \ref{periodic} implies that $\tilde{\varphi}=\mathrm{Id}_{\Sigma}$. Thus $\varphi$ acts fiberwise. 

Let $n$ be the order of $\varphi$, so $\varphi$ generates an effective $\mathbb{Z}_n$-action on $M$. To complete the argument, let $H \subset \mathbb{Z}_{n}$ be any  non-trivial subgroup. By Lemma \ref{inv}, Fix$(H)$ is an almost complex submanifold of $M$, with respect to the $H$-invariant, $\omega$-compatible almost complex structure $J$. Since the action is effective, Fix$(H) \neq M$. 

Observe that Fix$(H)$ intersects every fiber in at least two points. Therefore, Fix$(H)$ must contain a $J$-holomorphic curve that intersects the fiber positively. As Fix$(H)$ is an almost complex submanifold and contains such a curve, it cannot  contain any of the fibers. Hence, the $\mathbb{Z}_{n}$-action restricted to each fiber is effective. Applying this reasoning to each non-trivial element of $G$ shows that the group acts effectively on the fibers.

If $g(\Sigma) = 1$, define $$G_{1} : = \{g \in G | f \circ g = f\},$$ and set $G_{2} : = G/G_{1}$. Then there is a short exact sequence $$1 \rightarrow G_1 \rightarrow G \rightarrow G_{2} \rightarrow 1 .$$

The fact that $G_{1}$ acts effectively on the fibers follows by the same argument as in the case $g(\Sigma)>1$. Furthermore, by the commutative diagram above, $G_{2}$ acts effectively on $\Sigma$, and its induced action on $H_{1}(\Sigma,\mathbb{Z})$ is trivial. Hence, by Lemma \ref{torusgroups}, the isomorphism type of $G_{2}$ satisfies the claimed property.  \end{proof}

Example \ref{h1noth2} illustrates that Proposition \ref{bundle} applies to a broader class of symplectic group actions beyond the homologically trivial case. As we will see in Remark \ref{h12triv}, for irrational ruled symplectic $4$-manifolds, a finite group action is homologically trivial if and only if it acts trivially on both $H_{1}(M,\mathbb{Z})$ and $H_{2}(M,\mathbb{Z})$.

\subsection{Classification of finite groups acting on irrational ruled symplectic $4$-manifolds}\label{finclass}

 We now use Proposition \ref{bundle} to classify finite subgroups of symplectomorphisms. We first restrict our attention to symplectomorphisms that act trivially on homology.  

\begin{cor}\label{finact}
Let $(M,\omega)$ be an irrational ruled symplectic $4$-manifold whose base surface has genus $g 
\geq 2$. Suppose a finite group $G$ acts effectively by symplectomorphisms and trivially on $H_1(M, \mathbb Z)$. Then $G$ is isomorphic to a finite subgroup of $\mathrm{SO}(3)$. That is, $G$ is either cyclic, dihedral, or acts via the symmetries of a regular tetrahedron, octahedron, or icosahedron. 

Moreover, every such finite group admits a homologically trivial symplectic action on some irrational ruled symplectic $4$-manifold.
\end{cor}

\begin{proof}
By Proposition \ref{bundle}, the group $G$ preserves a $G$-invariant $\omega$-compatible almost complex structure $J$, and the associated fibration of $M$ by $J$-holomorphic spheres. Since each fiber is biholomorphic to $\mathbb{CP}^1$,  the group $G$ acts effectively by biholomorphisms on the fiber. The group of biholomorphisms of $\mathbb{CP}^1$ is the M\"obius group, and every finite subgroup of the M\"{o}bius group is conjugate to a subgroup of its maximal compact subgroup $\mathrm{SO}(3)$. The finite subgroups of SO$(3)$ are classically classified: they are either cyclic, dihedral, or the symmetry groups of the Platonic solids. 

For the existence part, consider $M=\Sigma \times S^2$, where $\Sigma$ is a surface of genus $g\ge 2$. Let $G$ act trivially on $\Sigma$ and via a standard rotation action on $S^2$. This yields a homologically trivial symplectic action of $G$ on $M$. 
\end{proof}

We remark that the classification of finite subgroups of symplectomorphisms is not known for rational $4$-manifolds, even in the case of $\mathbb {CP}^2$; see \cite{CLW} for partial results and further references.

It follows from Corollary \ref{finact} that the only finite abelian groups that can act homologically trivially on an irrational ruled symplectic $4$-manifold with base genus $g\ge 2$ are cyclic groups and the Klein four-group $\mathbb Z_2\times \mathbb Z_2$. We will later show that any homologically trivial $\mathbb Z_n$-action with $n\ge 3$ can be extended to a Hamiltonian $S^1$-action, possibly after modifying the symplectic form. However, this result does not hold in general for homologically trivial $\mathbb Z_2$-actions. 

We also prove the following statement for base genus $1$, which necessarily allows for non-trivial transformations on the base.

\begin{cor}\label{finacttorus}
Let $(M,\omega)$ be an irrational ruled symplectic $4$-manifold whose base surface has genus $g =1$. Suppose a finite group $G$ acts effectively by symplectomorphisms and trivially on $H_1(M, \mathbb Z)$. Then, there is an exact sequence $$1 \rightarrow G_1 \rightarrow G \rightarrow G_{2} \rightarrow 1 ,$$ such that: \begin{enumerate} \item  $G_1$ is isomorphic to a finite subgroup of $\mathrm{SO}(3)$; that is, $G_1 \cong \mathbb{Z}_{p}, D_{p}$, or one of the Platonic rotation groups (the symmetry group of a regular tetrahedron, octahedron, or icosahedron).  

\item There exist non-negative integers $p,q$ such that $$G_{2} \cong \mathbb{Z}_{p} \times \mathbb{Z}_{q}.$$
\end{enumerate}
\end{cor}

We record a structural consequence of Corollary \ref{finact}, describing the general form of any finite symplectomorphism group acting on an irrational ruled $4$-manifold with base genus $g\ge 2$.

\begin{cor}\label{finsym}
Let $(M,\omega)$ be an irrational ruled symplectic $4$-manifold whose base surface has genus $g 
\geq 2$, and let $G$ be a finite group acting effectively by symplectomorphisms on $M$. Then there exists a short exact sequence $$1 \rightarrow H_{1} \rightarrow G \rightarrow H_{2} \rightarrow 1,$$ where $H_{1}$ is a finite subgroup of $\mathrm{SO}(3)$, and $H_{2}$ is a finite subgroup of $\mathrm{SL}_{2g}(\mathbb Z)$.
\end{cor}

We now state the corresponding structural result when the base has genus $1$.

\begin{cor}\label{finsymone}
Let $(M,\omega)$ be an irrational ruled symplectic $4$-manifold whose base surface has genus $g 
=1$, and let $G$ be a finite group acting effectively by symplectomorphisms on $M$. Then there exist short exact sequences $$1 \rightarrow H \rightarrow G \rightarrow H' \rightarrow 1,$$

$$1 \rightarrow H_1 \rightarrow H \rightarrow H_{2} \rightarrow 1,$$
such that:
\begin{enumerate}
\item $H'\cong\mathbb{Z}_{k}$, where $k = 1,2,3,4,6$.

\item  $H_1$ is isomorphic to a finite subgroup of $\mathrm{SO}(3)$. That is, $H_1 \cong \mathbb{Z}_{p}, D_{p}$, or one of the Platonic rotation groups (the symmetry group of a regular tetrahedron, octahedron, or icosahedron).  

\item  There exist non-negative integers $p,q$ such that $$H_{2} \cong \mathbb{Z}_{p} \times \mathbb{Z}_{q}.$$
\end{enumerate}
\end{cor}
\begin{proof}
Define $$H := \{g \in G |g_*= \mathrm{Id}: H_{1}(M,\mathbb{Z}) \rightarrow   H_{1}(M,\mathbb{Z}  )\} ,$$ and set $H' = G/H$. Then  $H$ is finite  and acts by symplectomorphisms on $(M,\omega)$ while trivially on $H_{1}(M,\mathbb{Z})$.  The claimed exact sequence for $H$ then follows directly from Proposition \ref{bundle}b). 

Finally, the finite group $H'$ acts effectively on $H_{1}(T^2,\mathbb{Z}) \cong \mathbb{Z}^2$, and hence is isomorphic to a finite subgroup of $\mathrm{SL}_{2}(\mathbb{Z})$. The finite subgroups of $\mathrm{SL}_{2}(\mathbb{Z})$ are well known to be cyclic of order $1, 2, 3,4, 6$, so $H'\cong\mathbb{Z}_{k}$ for $k = 1,2,3,4,6$.
\end{proof}

In Example \ref{nonsplit} an algebraic action of the symmetric group $S_3$ on an ruled surface $S$ over an elliptic curve is given, showing that the first short exact sequence does not split in general.

In the preceding classification of finite subgroups of symplectomorphisms, we do not see how they would act on second homology group $H_2(M, \mathbb Z)$. The following classifies the induced action of $\mathrm{Symp}(M, \omega)$  on $H_2(M, \mathbb Z)$.

\begin{prop}\label{acth2}
For an irrational ruled $4$-manifold $(M, \omega)$, the induced action of $\mathrm{Symp}(M, \omega)$ on the second homology group is a subgroup of $SG(l)$ with $l=b_2-2$. 
\end{prop}
Here $SG(l)=(\mathbb Z/2\mathbb Z)^l \rtimes S_l$ is the group of even permutations $\phi$ of the set $\{-l, \cdots, -1, 1, \cdots, l\}$ satisfying $\phi(-k)=-\phi(k)$ for all $1\le k\le l$. This is the Weyl group of SO$(2l)$ and has rank $2^{l-1}l!$. 
\begin{proof}
By Theorem 1.8 of \cite{LW} (note that there is a typographical error there: it is stated for  the homological action of $\mathrm{Symp}(M, \omega)$),  the induced action of $\mathrm{Symp}(M, \omega)$ on the second homology group is generated by Lagrangian Dehn twists along Lagrangian spheres. For irrational ruled $4$-manifolds, the set of Lagrangian sphere classes is a subset of $\{\pm(F-E_i-E_j), \pm (E_i-E_j), 1\le i<j\le l\}$, which acts by even permutations on the set of $-1$-sphere classes $\{E_i, F-E_i, i=1, \cdots, l\}$. Here $F$ is the fiber class and $\{E_i\}$ is a maximal collection of orthogonal $-1$-sphere classes. This action is precisely the definition of $SG(l)$. 

In fact, the Lagrangian sphere classes form a root system which is the sub-root system of type $D_l$ generated by  $\{\pm(F-E_i-E_j), \pm (E_i-E_j), 1\le i<j\le l\}$, see Figure~\ref{fig:roots}. Therefore, the homological action of $\mathrm{Symp}(M, \omega)$ on $H_2(M, \mathbb Z)$ is the Weyl group of a sub-root system of this root system, namely a subgroup of $SG(l)$.

\begin{figure}[h] 
\begin{tikzpicture}[scale =1.3]

\draw[fill] (-0.3,0.7) circle [radius=0.05];

\draw[fill] (-0.3,-0.7) circle [radius=0.05];

\draw[fill] (1,0) circle [radius=0.05];

\draw[fill] (2.5,0) circle [radius=0.05];

\draw[fill] (4.5,0) circle [radius=0.05];

\draw[fill] (6,0) circle [radius=0.05];

\draw (1,0) -- (-0.3,0.7) ;

\draw (1,0) -- (-0.3,-0.7) ;

\draw (1,0) -- (2.5,0);

\draw[dashed] (2.5,0) -- (4.5,0);

\draw (4.5,0) -- (6,0);

\node at (-0.3,0.9) {$F - E_1 - E_2$};

\node at (1.28,0.25) {$E_2 - E_3$};

\node at (-0.3,-0.9) {$ E_1 - E_2$};

\node at (6,0.2) {$ E_{l-1} - E_{l}$};

\end{tikzpicture}
\caption{The complete root system of Lagrangian spheres with the induced homological action}
\label{fig:roots}
\end{figure}
\end{proof}

The classifications given in Theorem \ref{intro:finclass} and in Proposition \ref{acth2} are independent of each other. Indeed, for a finite group $G$ acting effectively by symplectomorphisms, the induced action on $H_2(M, \mathbb Z)$ is a subgroup of $SG(l)$ and can contribute to any component of the triple $\{H', H_1, H_2\}$. However, when we restrict to  symplectomorphisms acting trivially on $H_{1}(M,\mathbb{Z})$, the induced action on $H_2(M, \mathbb Z)$ can contribute only to the group $H_1$. In Example \ref{h1noth2}, we have $H_1=\mathbb Z_k$, which arises entirely from the induced action on $H_2(M, \mathbb Z)$, while the other terms vanish.

\subsection{Fixed point sets of symplectic cyclic actions} \label{fixedset}
In this subsection, we analyze the fixed point sets of symplectic cyclic actions. Our first result concerns a $\mathbb Z_k$-action with $k>2$, which will be used in the construction of the circle action in Section \ref{circle}.
\begin{prop} \label{sections}
Let $(M,\omega)$ be an irrational ruled symplectic $4$-manifold whose base surface has genus $g 
\geq 2$. Suppose $\varphi$ is a symplectomorphism of  order $k>2$, generating an effective homologically trivial $\mathbb Z_k$-action on $M$. Let $J$ be a $\mathbb Z_k$-invariant $\omega$-compatible almost complex structure, and let $f: M\rightarrow \Sigma$ be the associated the moduli map. Then:

\begin{enumerate} \item The fixed point set $M^{\mathbb{Z}_k}$ contains two disjoint sections with respect to the fibration $f: M \rightarrow \Sigma$.

\item The map $\varphi$ preserves each irreducible component of the fiber $f^{-1}(p)$ for all $p \in \Sigma$.
\end{enumerate}
\end{prop}
\begin{proof}

By Proposition \ref{bundle},   the moduli map $f: M \rightarrow \Sigma$ is $\mathbb Z_k$-equivariant, and each fiber is preserved by the $\mathbb{Z}_k$-action. 

Let $F = f^{-1}(p) \cong S^2$ be a smooth fiber. Since the almost complex structure $J$ is integrable on $F$, the induced action on the fiber is by a M\"{o}bius transformation of finite order, in particular, an elliptic transformation. Such a transformation fixes exactly two points, with weights $a \mod k$ and $-a \mod k$, where $\mathrm{gcd}(a,k)=1$ because the action is effective.  Since $k>2$, we must have $a \neq -a \mod k$, as otherwise $2a = 0 \mod k$,  contradicting the condition gcd$(a,k)=1$.

Therefore, the fixed point set $M^{\mathbb{Z}_k}$ intersects each smooth fiber in exactly two distinct points. Let $C\subset M^{\mathbb{Z}_k}$ be an irreducible component of the fixed point set that intersect the smooth fiber $F$. Then the intersection number $  C\cdot F =\int_{\Sigma} f_{*}([C]) > 0$.
By Lemma \ref{inv}, $C$ is a smooth $J$-holomorphic curve. The curves $C$ and $F$ intersect transversely because the fixed subspace of the tangent space at a fixed point is disjoint from the tangent space of the fiber away from the origin, as the action is non-trivial on $F$.  Therefore, $C\cdot F \leq 2$.

We now show that $C\cdot F =1$.  Each fixed point on $F$ lies in a connected component of the fixed point set $M^{\mathbb{Z}_k}$ of dimension $2$. The normal bundle to the fiber at a fixed point is the tangent space $T_{p}F$, and the $\mathbb Z_k$-action on the normal bundle has weight $w \in \mathbb{Z}_k$, constant along connected components. Since the weights at the two fixed points on $F$ are distinct, they must lie in distinct connected components of $M^{\mathbb{Z}_k}$. 
Hence, the two fixed points belong to two different $2$-dimensional components $\Sigma_1,\Sigma_{2} \subset M^{\mathbb{Z}_k}$, each intersecting $F$ transversely at a single point. Therefore both $\Sigma_1$ and $\Sigma_2$ are sections of the fibration,  proving part (1).

Now let $S$ be an irreducible component of a singular fiber. By Theorem \ref{fibration}, such components are $J$-holomorphic spheres with self-intersection $S \cdot S < 0$. Suppose,  for contradiction,  that $\varphi(S) \neq S$. Since $\varphi$ acts trivially on $H_{2}(M,\mathbb{Z})$,  we have $[\varphi(S)]=[S]$. Then by positivity of intersections $$ S \cdot S = \varphi(S) \cdot S \geq 0, $$  contradicting $S\cdot S<0$. Hence $\varphi(S)=S$, and each irreducible component of every fiber is preserved, proving part (2).
\end{proof}

Upon reviewing the proof of Proposition \ref{sections}, the reader will notice that the we only used the assumption that the action is trivial on $H_{1}$ and $H_{2}$. In the present context,  this is equivalent to the action being homologically trivial.

Let $N$ be an orientable, closed $4$-manifold with torsion-free homology, and let $h: N \rightarrow N$ be an orientation-preserving homeomorphism acting trivially on $H_{1}(N,\mathbb{Z})$. Let $$\mathrm{PD}: H^{3}(N,\mathbb{Z}) \rightarrow  H_{1}(N,\mathbb{Z})$$ denote the Poincar\'{e} duality isomorphism.  By the naturality of Poincar\'{e} duality for orientation-preserving homeomorphisms, we have $$ \mathrm{PD} = h_{*}\circ \mathrm{PD}\circ h^{*},$$ which implies that if $h_*$ acts trivially on $H_1(N, \mathbb Z)$,  then $h^*$  acts trivially on $H^{3}(N,\mathbb{Z})$. Since $N$ has torsion-free homology, we have $$H^{3}(N,\mathbb{Z}) \cong \mathrm{Hom}(H_{3}(N,\mathbb{Z}),\mathbb{Z}),$$ and therefore $h$ must act trivially on $H_{3}(N,\mathbb{Z})$. This proves the following:

\begin{remark} \label{h12triv}
Let $G$ be a group acting by orientation-preserving homeomorphisms on an orientable, closed $4$-manifold $N$ with torsion-free homology (in particular, any ruled symplectic $4$-manifold). Then the induced action of $G$ is homologically trivial if and only if it acts trivially on $H_{1}(N,\mathbb{Z})$ and $H_{2}(N,\mathbb{Z})$. 
\end{remark}

We now turn to the case of involutions. The following amendment of Proposition \ref{sections} holds for $k=2$, where an additional possibility arises:  unlike the case $k>2$, the fixed point set may be connected. In Theorem \ref{exampleir}, we will construct an example where this occurs, providing a key ingredient for a counterexample to Question \ref{kedra} in the case of involutions on minimal irrational ruled $4$-manifolds. 

 \begin{prop}\label{z2fix}
Let $(M,\omega)$ be an irrational ruled symplectic $4$-manifold with $b_2=2$, with base genus $g \geq 2$. Suppose $\tau$ is a non-trivial symplectic involution acting trivially on $H_{1}(M,\mathbb{Z})$. Then the fixed point set $M^{\tau}$ is  either: \begin{enumerate}
\item A connected bisection of genus $2g-1$ with self-intersection $0$, intersecting all fibers transversely, or
\item Two disjoint sections of genus $g$, whose self-intersection numbers sum to $0$. 
\end{enumerate} 
\end{prop}
\begin{proof}
As in Proposition \ref{bundle}, the involution $\tau$ preserves the fibration structure and acts fiberwise. Since $M^{\tau}$ intersects each fiber in two points, it forms an almost complex submanifold that intersects fibers transversely in two points, so $$M^{\tau}\cdot F=2 .$$ In particular,  $M^{\tau}$ cannot contain a fiber as an irreducible component.  
Decomposing $M^{\tau}$ into connected components, it follows that it must consist of either one or two components, which is a bisection or two  disjoint sections.

The projection of each component to the base surface $\Sigma$ is either a homeomorphism (in the case of a section) or a two-to-one covering (in the case of a bisection). This gives the claimed genera:
\begin{itemize}
\item For a section: genus $g$,
\item For a bisection: genus $2g-1$.
\end{itemize}

Since $M$ is an $S^2$-bundle over a surface, the signature $\sigma(M) = 0 $.  By the main result of \cite{JO}, the self-intersection of the fixed point set satisfies: 

$$M^{\tau}\cdot M^{\tau} =  \sigma(M) = 0, $$ which establishes the claims about the self-intersections.
\end{proof}

From the above argument, if the condition $b_2=2$ is removed, the same topological classification of $M^{\tau}$ holds, though the self-intersection calculation no longer applies.

\section{Exotic symplectic involutions on irrational ruled symplectic $4$-manifolds}
\label{exoticsymplecticinv}

It is well known that the fixed point set of a homologically trivial symplectic circle action on a minimal irrational ruled symplectic $4$-manifold with base genus at least two consists of a disjoint union of two sections. The same conclusion holds for a homologically trivial $\mathbb Z_k$-action when $k>2$, as shown in Proposition \ref{sections}. However, we will now demonstrate that this structure breaks down in the case of a symplectic involution ({\it i.e.} $k=2$). 
\begin{theorem} \label{exampleir}
For any compact orientable surface $S$ of genus at least one, there exists:
\begin{itemize}
\item An involution $q$ of $S \times S^2$, 
\item An invariant symplectic form $\omega_0$,
\end{itemize}
such that the fixed point set of $q$ is a connected, smooth bisection. The symplectic form can be chosen to have cohomology class $$[\omega_0] = \mathrm{PD}([S]) + K \cdot \mathrm{PD}([S^2]),$$ for $K\gg 0$,  where $\mathrm{PD}$ denotes Poincar\'{e} dual.
Moreover, the $\mathbb{Z}_2$-action generated by $h$ is homologically trivial and induces the trivial action on $\pi_1$. 
\end{theorem}
\begin{proof}

Let $S$ be a compact orientable surface of genus at least one. Then $S$ admits a double cover $\tilde{S}\rightarrow S$, with covering involution $\tau$.  There exists an embedded circle $C \subset \tilde{S}$ invariant under $\tau$, and a $\tau$-equivariant retraction $R: \tilde{S} \rightarrow C$  onto $C$, where $\tau$ acts on $C\cong S^1$ by rotation through angle $\pi$.

This setup can be realized geometrically by embedding $\tilde{S}$ into $\mathbb R^3$ so that $\tau$ acts as rotation about the $z$-axis with angle $\pi$. Let $C$ be a circle in the horizontal plane $\{z=0\}\subset \mathbb R^3$,  centered at $(0,0,0)$,  with radius $1$ (see Figure \ref{example}). Then the retraction $$R: \mathbb{R}^3 \setminus L \rightarrow C\cong S^1, \quad R(x,y,z) = (\frac{x}{|(x,y)|},\frac{y}{|(x,y)|}, 0),$$
restricts to a $\tau$-equivariant retraction map $R|_{\tilde{S}}: \tilde{S}\rightarrow C$. 

\begin{figure}[h] 
\begin{tikzpicture}[scale=0.4]

\draw[gray, thick] (-10,0)  edge[in=90,out=90] (10,0);

\draw[gray, thick] (10,0)  edge[in=-90,out=-90] (-10,0);

\draw[thick,blue] (-1.1,-0.2)  edge[in=90,out=90] (1.1,-0.2);

\draw[thick, dashed,blue] (1.1,-0.2)  edge[in=-90,out=-90] (-1.1,-0.2);

\draw[ thick,->] (0,-0.6) -- (0,8);

\node at (-1.2,1) {$C$};

\node at (-7,5) {$\tilde{S}$};

\node at (2.2,7) {$L$=$z$-axis};

\draw[ thick, dashed] (0,-5.8) -- (0,-0.6);

\draw[ thick] (0,-6) -- (0,-7);

\draw (-2,0.5) .. controls (0,-1) .. (2,0.5);

\draw (-0.75,-0.35) .. controls (0,0.15) .. (0.75,-0.35);

\draw (-8,0.5) .. controls (-6,-1) .. (-4,0.5);

\draw (-6.75,-0.35) .. controls (-6,0.15) .. (-5.25,-0.35);

\draw (8,0.5) .. controls (6,-1) .. (4,0.5);

\draw (6.75,-0.35) .. controls (6,0.15) .. (5.25,-0.35);
\end{tikzpicture}
\caption{The involution on $\tilde{S}$ and the invariant circle $C$.}\label{example}
\end{figure}


Define a map $$\mu: \tilde{S} \rightarrow S \times S^2, \quad \mu(X) = (\Pi(X), R(X)),$$where $\Pi : \tilde{S} \rightarrow S$ is the covering map.

The map $\mu$ is an embedding, for suppose $\Pi(X) = \Pi(Y)$ for $X \neq Y$, then $X,Y$ are in the same orbit of the $\tau$ hence $R(X) \neq R(Y)$, by the equivariance property of $R$. Set $$S_{0} := \mu (\tilde{S}) \subset S \times S^2 .$$ Then $S_0$ is a smooth embedded bisection intersecting each $S^2$-fiber in two antipodal points. 

Define an involution $q\in \mathrm{Diff}(S\times S^2)$ such that, for each fiber $F\cong S^2$, $q|_{F}$ is the M\"obius transformation rotating by angle $\pi$ about the two points $S_{0} \cap F$. This $q$ is a smooth involution of $S \times S^2$ with the fixed point set equal to $S_{0}$.

To construct a symplectic form invariant under $q$, start with the product symplectic form $\omega$ on $S \times S^2$, and define: $$\alpha = \frac{1}{2}(\omega + q^{*}\omega).$$
Then $\alpha$ is closed and $\mathbb{Z}_{2}$-invariant. Moreover, since $q$ acts symplectically on each fiber, $\alpha|_{F} = \omega|_{F}$ for every fiber $F\cong S^2$. 

Let $\beta$ be the pullback of a symplectic form on $S$ via the projection $\pi_1: S \times S^2\rightarrow S$. It is closed.
Then, for any $K>0$, the form $$\omega_0:=\alpha + K \beta ,$$ is closed, $\mathbb{Z}_2$-invariant, and non-degenerate for sufficiently large $K$. Thus $\omega_0$ is a $\mathbb Z_2$-invariant symplectic form on $S\times S^2$.

Since $q$ preserves a symplectic form, it is orientation-preserving.  The homology classes $[S_{0}]$ and $[F]$ are both preserved under the action of $\mathbb Z_2$, and together they span $H_{2}(S\times S^2, \mathbb R)$. Therefore, the $\mathbb Z_2$-action is trivial on $H_{2}$. 

Now consider the projection $\pi_{1} : S \times S^2 \rightarrow S$. Recall that  $\pi_{1} \circ \mu = \Pi: \tilde S\rightarrow S$. Hence, the induced map $$\mu_{*} : H_{1}(\tilde{S},\mathbb{R}) \rightarrow H_{1}(S,\mathbb{R})$$ is surjective. Since $S_0=\mu(\tilde S)$ is fixed by $q$, and $H_1(S_0, \mathbb R)$ surjects onto $H_1(S, \mathbb R)\cong H_1(S\times S^2, \mathbb R)$, it follows that the involution $q$ acts trivially on $H_{1}(S\times S^2, \mathbb R)$.  By Remark \ref{h12triv}, $q$ is homologically trivial.

Finally, to see that the action on the fundamental group is trivial, let $\{\alpha_i\}$  be a standard  generating set of $\pi_{1}(S)$, represented by $2g$ simple closed curves. Then the loops  $\alpha_i \times \{p\}$ generate $\pi_{1}(S \times S^2)$; we denote the corresponding classes by $[\alpha_i]$.

For each $i$, define the submanifold $N_i :=  \alpha_i \times S^2\subset S\times S^2$. Each $N_i$ is invariant under the involution $q$,  and the inclusion map $i_{*} : \pi_{1}(N_{i}) \rightarrow \pi_{1}(S\times S^2)$ is $\mathbb{Z}_2$-equivariant. The fixed point set of the restricted action on $N_{i}$ is a smooth bisection of the projection to $\alpha_i$. Therefore, the $\mathbb{Z}_2$-action on $\pi_{1}(N_{i})\cong  \mathbb{Z}$ fixes the subgroup $2\mathbb{Z} \subset \mathbb{Z} $, and hence must fix all of $\pi_{1}(N_{i})$. By equivariance of $i_{*}$, this implies that $q$ fixes $[\alpha_i]$. Since the $[\alpha_i]$ generate $\pi_{1}(S \times S^2)$, the $\mathbb Z_2$-action is trivial on the fundamental group. 
\end{proof}

The example constructed in Theorem \ref{exampleir}  provides a negative answer to Question \ref{kedra'} for irrational ruled $4$-manifolds with base genus at least two and any $b_2\ge 2$. 
\begin{cor}\label{invtocircle}
There exist homologically trivial symplectic involutions on irrational ruled $4$-manifolds with any base genus $g\ge 2$ and any $b_2\ge 2$, which cannot be extended to symplectic circle actions, even after possibly modifying the symplectic form. 
\end{cor}
\begin{proof}
For an effective symplectic circle action, if any non-trivial element in this action acts trivially on homology, then all non-trivial elements must also be homologically trivial, because the action on $H_i(M, \mathbb Z)$ is locally constant. Therefore, if such an extension exists,  we may assume that the full $S^1$-action is homologically trivial. 

However, it is known (see Proposition \ref{sections}) that any homologically trivial symplectic circle action on an irrational ruled $4$-manifold with base genus at least two must have two disjoint sections as its fixed point set. In Theorem \ref{exampleir},  we constructed an involution on the minimal irrational ruled $4$-manifold $(\Sigma_g\times S^2, \omega)$ with a connected bisection as the fixed point set, so the involution cannot extend to a homologically trivial symplectic circle action, with respect to any symplectic structure. 

To obtain examples with $b_2> 2$, we perform $\mathbb Z_2$-equivariant symplectic blowups of the above minimal example. This yields a fixed point set consisting of a bisection together with isolated points. But the fixed point set of a homologically trivial symplectic $S^1$-action must always contain two disjoint sections. Therefore, such involutions remain non-extendable to Hamiltonian or symplectic circle actions even after blow-ups. 
\end{proof}

We refer to these non-extendable  homologically trivial periodic symplectomorphisms as exotic. 
For rational or ruled $4$-manifolds, cohomologous symplectic forms are symplectomorphic, so our examples also yield exotic involutions on product symplectic manifolds of the form  $(\Sigma_g, \omega_g)\times (S^2, \omega_{\mathrm{FS}}),$ where $\omega_g$ is an area form with total area $K>0$. 

Moreover, our construction also provides exotic examples when the base genus is one. Consider the symplectic involution on  $T^2\times S^2$ constructed  in Theorem \ref{exampleir}, and perform an  equivariant blow-up to obtain irrational ruled $4$-manifolds with torus fixed point component and isolated points. On one hand, a Hamiltonian circle action on such a manifold must have at least two fixed components, each with fundamental group $\mathbb Z\oplus \mathbb Z$. On the other hand,  both components in this example are contained within a single torus, which contradicts to the required structure of the fixed point set.  Hence, the involution in this genus one case is not extendable to a Hamiltonian circle action.

The construction of Theorem \ref{exampleir} could be extended to higher dimensions. 

\begin{prop}\label{highex}
Let $G$ be a finitely presented group with an index $2$ subgroup $H\subset G$ and a surjective homomorphism $\phi: G \rightarrow \mathbb{Z}$ such that $\phi(H)=2\mathbb Z$. 
Then for any closed symplectic manifold $M$ with $\pi_{1}(M) \cong G$  and  $\dim(M) \geq 4$, let $\tilde{M}\rightarrow M$ be the double cover corresponding to the subgroup $H$. 

Then, there exists a symplectic involution on the product $M \times S^2$ whose fixed point set is diffeomorphic to $\tilde{M}$. 
Moreover, if the integral homology of $M$ is torsion-free, then the involution is homologically trivial. 
\end{prop}
\begin{proof}
Since $H$ is a subgroup of index $2$, it is automatically normal. The corresponding covering map $\pi : \tilde{M} \rightarrow M$ is thus a normal double cover, and the deck transformation group $\mathrm{Deck}(\pi) \cong \mathbb{Z}_2$ acts freely and transitively on the fibers.  From now on, we denote this $\mathbb Z_2$-action on $\tilde M$ by $\tau$.


We aim to construct a $\mathbb{Z}_2$-invariant circle $C \subset \tilde{M}$ and a $\mathbb{Z}_2$-equivariant retraction $R: \tilde{M} \rightarrow C$. The construction proceeds by first building a retraction of $M$ to a circle and lifting it to $\tilde{M}$.

Since $\phi: G\rightarrow \mathbb Z$ is surjective, pick an element $x \in G$  such that $\phi(x)=1$. Because $\dim(M) \geq 4$, there exists an embedded simple closed curve $C'\subset M$ representing the homotopy class $x\in \pi_1(M)$. Let $r: M \rightarrow C'\cong S^1$ be a continuous map corresponding to classifying map of $\phi$, which exists because $S^1$ is a $K(\mathbb{Z},1)$-space.


Let $T(C')\subset M$ be a closed tubular neighborhood of $C'$. Since the pair $(M,T(C'))$ has the homotopy extension property (as a CW pair), and $T(C')$ is also a $K(\mathbb{Z},1)$, we can homotope $r|_{T(C')}$  to $Id|_{T(C')}$.  Applying homotopy extension property to the map which is equal to $r$ on $M  \times \{0\} $ and is equal to the aforementioned homotopy from $r|_{T(C')}$ to the identity on $ T(C') \times [0,1] $, producing a homotopy $W: M \times [0,1] \rightarrow T(C')$. Set $r' = W|_{M \times \{1\}}$. By smoothing results \cite[Theorem 6.26]{L}, $r'$ is homotopic to a smooth map equal to the identity on a smaller closed tubular neighborhood $T'(C')\subset T(C')$\footnote{\cite[Theorem 6.26]{L} is a variant of the Whitney approximation theorem, proved by embedding the relevant smooth manifolds into Euclidean space and applying approximation results in analysis.}. Composing with the projection $T'(C') \rightarrow C'$, we obtain the required smooth retraction $$R' : M \rightarrow C'.$$ 


Now consider a circle $C:=\pi^{-1}(C')\subset \tilde M$,  which is a double cover of $C'$. The composition $R' \circ \pi : \tilde{M} \rightarrow C'$ induces a homomorphism on $\pi_1(\tilde M)$ whose image lies in $2\mathbb{Z}\subset \mathbb Z$. By the lifting criterion for the covering spaces, this lifts to a map: $$R: \tilde{M} \rightarrow C.$$  After possibly precomposing $R$ with the non-trivial deck transformation $\tau$, we may assume $R$ is $\mathbb Z_2$-equivariant with respect to $\tau$.

We now follow the same construction  as in Theorem \ref{exampleir}. Define the standard embedding  $i: S^1 \rightarrow S^2$ by $$i(x,y) = (x,y,0).$$ Then define an embedding: $$\mu : \tilde{M} \rightarrow M \times S^2, \quad \mu(p) = (\pi(p), i \circ R (p)).$$ Then image $\mu(\tilde{M})\subset M\times S^2$ is a smooth bisection intersecting each $S^2$-fiber in a pair of antipodal points. 

Define an involution on $M\times S^2$ by acting as the identity on $M$ and as rotation by $\pi$ about the equator on each fiber $S^2$, with fixed points the intersection  $\mu(\tilde{M})\cap (M\times S^2)$. The construction of the $\mathbb Z_2$-invariant symplectic form on $M\times S^2$ follows exactly as in Theorem \ref{exampleir} and is omitted. 

The last claim follows from the K\"{u}nneth theorem, which implies that  $H_{i}(M \times S^2 ,\mathbb{R})$ is generated by classes of the form $N \times \{p\}$ and $N' \times S^2$, where $N\subset M$ is an $i$-cycle and $N'\subset M$ is an $(i-2)$-cycle. The induced map on homology $\pi_{*}: H_{i}(\tilde M,\mathbb{R}) \rightarrow H_{i}(M,\mathbb{R})$ is surjective for all $i$ (see \cite[Lemma 1.2]{R}
).  It follows that any class of the form $N \times \{p\}$ is preserved by the involution. 

For classes of the form $N' \times S^2$, note that these are preserved on the level of singular chains by the product structure and the fact that the involution acts trivially on $M$ and by an orientation-preserving homeomorphism on $S^2$. Hence, the homology classes are preserved as well. 


 Finally, since the involution is trivial on a basis of $H_{i}(M \times S^2,\mathbb{R})$, and  because each homology group $H_{i}(M \times S^2,\mathbb{Z})$ is torsion-free, it follows that the involution acts trivially on $H_{i}(M \times S^2,\mathbb{Z})$ for all $i$. \end{proof}

The following specialization of the previous result gives a general recipe for constructing exotic examples in dimension $6$, starting from any finitely presented group $G$ satisfying  mild conditions. 
\begin{cor}\label{6dexample}
Let $G\ncong \mathbb Z\oplus \mathbb Z$ be a finitely presented group, such that its abelianization $G_{ab}= G/[G,G]$ is torsion-free, and suppose there exists a surjective homomorphism $\phi : G \rightarrow \mathbb{Z}$. Then, there exists a  closed symplectic $4$-manifold $M$, with $\pi_1(M) \cong G$, such that $M \times S^2$ admits a homologically trivial symplectic involution with non-empty fixed point set that does not extend to a Hamiltonian $S^1$-action.
\end{cor}
\begin{proof}
The existence of a closed symplectic $4$-manifold with $\pi_{1}(M) = G$ follows from \cite[Theorem 1.1]{Gom}. Let $H= \phi^{-1}(2\mathbb{Z})$, a subgroup of index $2$ in $G$.  By the universal coefficient theorem, the assumption that $G_{ab} = H_{1}(M,\mathbb{Z})$ is torsion-free implies that all homology groups of $M$ are torsion-free. Then, by Proposition \ref{highex}, $M\times S^2$ admits a homologically trivial symplectic involution with a non-empty fixed point set.

Now assume $G$ is not a surface group. Suppose, for contradiction, that the involution extends to a Hamiltonian $S^1$-action. Then the fixed point set must include at least two connected components corresponding to the maximum and minimum levels of the moment map.  Since these components are symplectic submanifolds of the connected $4$-dimensional symplectic manifold $\tilde M\subset M\times S^2$ corresponding to the subgroup $H$, they cannot be $4$-dimensional. Hence, they are necessarily symplectic surfaces whose fundamental groups are isomorphic to $\pi_1(M)\cong G$. But this contradicts the assumption that $G$ is not a surface group. 

In the case where $G$ is a non-abelian surface group, the construction and the non-extendability claim are addressed in Corollary \ref{main1'}.
\end{proof}

\subsection{Exotic actions of the Klein 4-group on irrational ruled symplectic $4$-manifolds}

Using the symmetry among different embeddings of $S^1$ into $S^2$ relative to a pair of antipodal points, we can construct exotic, homologically trivial symplectic actions of the Klein $4$-group on irrational ruled $4$-manifolds diffeomorphic to $\Sigma_g\times S^2$, for any genus $g\ge 1$. 

Recall that the Klein $4$-group is the only non-cyclic abelian subgroup that can arise as a group of homologically trivial symplectomorphisms on irrational ruled symplectic $4$-manifolds. This makes it a natural next case to investigate beyond the cyclic examples constructed earlier.  

\begin{theorem}\label{exoticklein}
For every integer $g\ge 1$, there exists a homologically trivial symplectic action of the Klein $4$-group $\mathbb Z_2\times \mathbb Z_2$ on an irrational ruled $4$-manifold diffeomorphic to $\Sigma_g\times S^2$, such that two out of the three non-trivial involutions in the group action cannot be extended to a symplectic circle action, even after possibly modifying the symplectic form. 

Moreover, this Klein $4$-group action cannot be extended to a symplectic action of $\mathrm{SO}(3)$,  with respect to any embedding $\mathbb Z_2\times \mathbb Z_2\subset \mathrm{SO}(3)$.  
\end{theorem}
\begin{proof}
The idea is to apply Theorem \ref{exampleir} to construct three involutions. For each $j=1,2,3$, let $i_j: S^1 \rightarrow S^2$ be such that the associated embeddings $\mu_1, \mu_2, \mu_3$ have pairwise disjoint images. The resulting involutions have fixed point sets given by Im$(\mu_j)$, respectively. To ensure that the these involutions and the identity give an action of $\mathbb{Z}_2 \times \mathbb{Z}_2$ on  $\Sigma_g\times S^2$, we have to ensure that the images of $i_j$ form an orthonormal frame of $\mathbb{R}^3$.


This may be done as follows. Let $$i_1(\cos\theta, \sin\theta)=(\cos\theta, \sin\theta, 0),$$ $$i_2(\cos\theta, \sin\theta)=(-\sin\theta, \cos\theta, 0)$$ and $i_{3}(\cos\theta, \sin\theta) = (0,0,1)$. The associated involutions $q_1$ and $q_2$ are exotic in the sense that they cannot be extended to symplectic circle actions, the proof is as in Corollary \ref{invtocircle} using the fact that fixed point set is connected. However, the third involution $q_3$, generated by composing $q_1$ and $q_2$, is not exotic.  Its fixed point set in each $S^2$-fiber is the standard pair $(0, 0, \pm 1)$, and $q_3$ acts as the standard product involution $\mathrm{id_{\Sigma}}\times r$ of $\Sigma_g\times S^2$, where $r$ is the $180^{\circ}$ rotation about the $z$-axis on $S^2$ .

To construct a symplectic form invariant under this group action, we average the standard product symplectic form $\omega$ over the group: $$\alpha = \frac{1}{4}(\omega + q_1^{*}\omega+ q_2^{*}\omega+ q_3^{*}\omega).$$
Since each $q_j$ is a symplectomorphism on fibers and acts trivially on the base, $\alpha$ is closed and non-degenerate on the fibers. Let $\beta$ be the pullback of an area form on $\Sigma_g$. Then for sufficiently large $K\gg 0$,  the form $$\omega'=\alpha+K\beta$$ is a symplectic form invariant under the Klein $4$-group action.

Note that the fixed point sets of the involutions associated to $i_1$ and $i_{2}$ are bisections. It follows from Corollary \ref{invtocircle} that involutions $q_1,q_2$ cannot extend to a symplectic circle action, even after changing the symplectic form. 

For the last statement, suppose the Klein $4$-group action could extend to a symplectic action of SO$(3)$ with some embedding $\mathbb Z_2\times \mathbb Z_2\subset \mathrm{SO}(3)$. Then  any nontrivial element $a\in \mathbb Z_2\times \mathbb Z_2\subset \mathrm{SO}(3)$ would lie on  some maximal torus $S^1\subset \mathrm{SO}(3)$. But as shown above, two such $a$ cannot extend to a symplectic circle action. This contradiction shows the Klein $4$-group action does not extend to SO$(3)$ symplectically. 
\end{proof}

In Section \ref{dihedral}, some actions of a dihedral group $D_{2k}$ on minimal irrational ruled surfaces are exhibited, for which the fixed set of some of the elements is a bisection and which preserve integrable complex structures of a certain type.

\section{Extension of finite symplectic actions on irrational ruled $4$-manifolds} \label{circle}

In this section, we mainly focus on the case where $(M,\omega)$ is an irrational ruled symplectic $4$-manifold with $b_{2}(M)=2$ and  base genus $g 
\geq 2$. Suppose that $f : M \rightarrow M$ is a symplectomorphism generating a homologically trivial $\mathbb{Z}_k$-action with $k>2$. We will demonstrate that, in contrast to the case of $\mathbb{Z}_2$ case, where exotic actions may occur, any such $\mathbb Z_k$-action with $k>2$  extends to a Hamiltonian circle actions. This is Theorems \ref{main2}  stated in the introduction.

 Before proceeding, we take a brief detour to recall some elementary facts about M\"{o}bius transformations of $\mathbb{CP}^1$. The key fact is that there is a canonical way to extend an effective $\mathbb{Z}_{n}$-action on $\mathbb{CP}^1$ that preserves its complex structure to an $S^1$-action on $\mathbb{CP}^1$ that also preserves the complex structure, which we now describe.

Let $\alpha : \mathbb{CP}^1 \rightarrow \mathbb{CP}^1$ be the M\"{o}bius transformation generating the $\mathbb{Z}_{n}$-action. Since $\alpha$ has finite order, it is elliptic, and hence, in suitable coordinates, is given by the map $$z \mapsto e^{\frac{2\pi i}{n}} z.$$ We denote the two fixed points of $\alpha$ by $\alpha_0$ and $\alpha_{\infty}$,  corresponding respectively to $0$ and $\infty$ in these coordinates. This $\mathbb{Z}_{n}$-action naturally extends to an $S^{1}$-action by $$e^{2\pi i \theta} \cdot z = e^{2\pi i \theta} z.$$ 

This discussion may be summarized in the following key fact:

\textbf{Key fact.} Given an \textit{ordered} pair of points $p_{1},p_{2} \in \mathbb{CP}^1$ and  angle $\theta \in [0,2\pi)$, there exists a unique elliptic M\"{o}bius transformation $\alpha : \mathbb{CP}^1 \rightarrow \mathbb{CP}^1$ with fixed point set $\{p_{1},p_{2}\}$, such that the induced action on the tangent space at $p_{1}$ is multiplication by $e^{2 \pi i \theta}$.

This is the central idea underpinning our extension process. By Proposition \ref{sections},  the fixed point set of the $\mathbb Z_k$-action on $M$ consists of two disjoint sections, and thus  each fiber contains two fixed points with a natural ordering.  
Moreover, the associated M\"{o}bius transformations depend smoothly on both the choice of fixed points and the rotation angle. 

To make this dependence explicit, let $V \subset S^2 \times S^2$ denote the open subset consisting of ordered pairs of distinct points ({\it i.e.}, the complement of the diagonal).
 Then there exists a smooth map $$\phi_{\alpha} : V \times S^1 \times S^2 \rightarrow S^2,$$ to extend the finite order M\"obius transformation $\alpha$, such that  for fixed $(p_1, p_2)\in V$ and $\theta\in S^1$, the restriction $\phi_{\alpha}|_{(p_1,p_2) \times \{\theta\} \times S^2}  $ is the elliptic M\"{o}bius transformation with angle $\theta$ and fixed points $p_1$, $p_{2}$,  acting on the tangent space at $p_{1}$ by multiplication with $e^{2 \pi i \theta}$. Furthermore, $$\phi_{\alpha}|_{(\alpha_0,\alpha_{\infty}) \times \{\frac{2\pi}{n}\} \times S^2}=\alpha.$$

We now return to the primary situation of interest, where $M$ is an irrational ruled symplectic $4$-manifold with $b_{2}(M) =2$ and  base genus at least $2$, to prove Theorem \ref{main2}.
\begin{theorem} \label{theorem}

Let $(M,\omega)$ be an irrational ruled symplectic $4$-manifold with $b_{2}(M)=2$, and suppose the base has genus $g 
\geq 2$. Let $f : M \rightarrow M$ be a homologically trivial symplectomorphism generating an effective $\mathbb{Z}_k$-action with $k>2$. Then  the $\mathbb{Z}_{k}$-action extends to a Hamiltonian $S^1$-action for some possibly different symplectic form $\omega'$, under the standard embedding $\mathbb{Z}_k \subset S^1$.

\end{theorem}
Here, the standard embedding $\mathbb{Z}_k \subset S^1$ means that the generator $f$ corresponds to the positive rotation by $\frac{2\pi}{k}$.
\begin{proof}
Let $f: M \rightarrow \Sigma$ be the moduli map, and $J$ the $\mathbb Z_k$-invariant,  $\omega$-compatible almost complex structure provided by Proposition \ref{bundle}. The assumption  $b_{2}(M)=2$ implies that every fiber is smooth. By Proposition \ref{sections}, there exist two sections, $\Sigma_{1},\Sigma_{2} : \Sigma_{g} \rightarrow M$, whose images consist of fixed points of the $\mathbb{Z}_k$-action. 

We define the $S^1$-action on the fiber $F$ containing a given point $p \in F$ as follows: for each $\theta \in S^1$, the action is given by the elliptic M\"{o}bius transformation with fixed points $\Sigma_1(p):=\Sigma_1(\pi(p))$ and $\Sigma_2(p):=\Sigma_2(\pi(p))$, such that the induced action on the tangent space of the fiber at $\Sigma_1(p)$ is multiplication by $e^{2 \pi i \theta}$. As described in the key fact above, this extension is unique.

After changing coordinates in a small neighbourhood $U \times S^2$ around a fiber $F$, we may assume that the restriction $J|_{F}$ is the standard almost complex structure. This coordinate change is given by the evaluation map (see Definition \ref{evmap} (c)). For the fact that the evaluation map is a local diffeomorphism in a neighborhood of  smooth fibers of the moduli map, see \cite[Page 185]{MS}. 
 In such a coordinate chart, the group action is defined by: $$\theta\cdot (u,p) =  (u, \phi_f((\Sigma_1(p),\Sigma_{2}(p)), \theta, p)),$$
where $\phi_f$ is defined immediately above the theorem.

The action of $\mathbb{Z}_k \subset S^1$ agrees with the original $\mathbb{Z}_k$-action since, by definition,  $$\phi_f((\Sigma_1(p),\Sigma_{2}(p)), \frac{2\pi}{k}, p)=f.$$ Since $f$ is a diffeomorphism, for any small $\epsilon$, the map $$f_{\epsilon}:=\phi_f((\Sigma_1(p),\Sigma_{2}(p)), \frac{2\pi}{k}+\epsilon, p)$$ is also a diffeomorphism. Choosing $\epsilon$ to be an irrational multiple of $\pi$, the map $f_{\epsilon}$ generates a dense orbit in $S^1$, which implies that the full group action $\theta\cdot (u,p)$ is a diffeomorphism for all $\theta\in S^1$. Hence, the group action is smooth. 

It remains to show that the circle action is Hamiltonian with respect to some symplectic form. We  first construct an $S^1$-invariant symplectic form. Let $\alpha$ be the average of $\omega$ with respect to the above constructed $S^1$-action. Then $\alpha$ is an $S^1$-invariant, closed $2$-form whose restriction to the the fibers' tangent spaces coincide with that of $\omega$.

Let $\beta$ be the pullback of any symplectic form on the base. Then the form $\omega' = \alpha + K \beta$ is closed and $S^1$-invariant for any $K>0$. For sufficiently large $K$, $\omega'$ is non-degenerate and thus symplectic.  Since this symplectic $S^1$-action has fixed points  $\Sigma_1$ and $\Sigma_2$, by \cite[Proposition 5.1.7]{MS}, the action on $(M,\omega')$ is Hamiltonian. 
\end{proof}

We remark that, in a different context, a similar averaging technique based on Thurston's construction, for producing invariant symplectic forms on a symplectic fibration with a group action, was used in \cite[Lemma 4.10]{W}. 

The same argument applies in the case where $b_{2}(M)=3$. In this setting, the fibration of $M$ by $J$-holomorphic spheres has only one singular fiber, which is a chain of two spheres. Their intersection point $x$ is the only fixed point of $f$ besides the two sections $\Sigma_1$ and $\Sigma_2$. As before, we can extend the $\mathbb Z_k$-action generated by $f$ to a smooth $S^1$-action such that, for each angle, the map is a diffeomorphism. On the singular fiber, this action corresponds to a combination of two elliptic diffeomorphisms of spheres, with fixed points  $\{\Sigma_1(x), x\}$ and  $\{x, \Sigma_2(x)\}$, respectively.

Section 6.4 of \cite{CK} explains how their construction of a homologically trivial symplectic involution on an irrational ruled symplectic $4$-manifold with $b_2(M)=5$ can be generalized to a $\mathbb Z_{2k}$-action on an irrational ruled symplectic $4$-manifold with $b_2(M)=2k+1$ for $k>1$. However, the argument in Section 6.3 of \cite{CK} shows that  these $\mathbb Z_{2k}$-actions do not extend to  Hamiltonian circle actions, even after possibly modifying the symplectic form. 

Hence, for each even integer $2k>2$, it is currently unknown whether there exist  examples of non-extendable homologically trivial $\mathbb Z_{2k}$-actions on irrational ruled symplectic $4$-manifolds with $4\le b_2(M)\le 2k$. In particular, do such examples exist when $b_2=4$? What about  cyclic group actions of odd order, {\it i.e.}, $\mathbb Z_{2l+1}$-actions?

The following corollary is a classification of homologically trivial symplectomorphisms of finite order $k$ greater than $2$ up to $\mathbb{Z}_k$-equivariant diffeomorphism. 

\begin{cor} \label{theoremtwo}
Let $(M,\omega)$ be an irrational ruled symplectic $4$-manifold with base genus $g 
\geq 2$. Let $f : M \rightarrow M$ be a homologically trivial symplectomorphism generating an effective $\mathbb{Z}_k$-action with $k>2$.

 Then there is a  minimal, ruled, symplectic $4$-manifold $\tilde{M}_{0}$ with a Hamiltonian $S^1$-action with a subgroup $\mathbb{Z}_k$-action and a $4$-manifold $\tilde{M}$  obtained by blowing up $\tilde{M}_{0}$ $\mathbb{Z}_k$-equivariantly along a sequence of $\mathbb{Z}_k$-fixed points such that $M$ is $\mathbb{Z}_k$-equivariantly diffeomorphic to $\tilde{M}$.
\end{cor}
\begin{proof}
Suppose that $b_{2}(M)>2$. Then $M$ contains a $-1$-sphere  $S$. Since $f$ acts trivially on $H_{2}(M,\mathbb{Z})$, we have $$S\cdot S = f(S)\cdot S = -1.$$ By the positivity of intersections of $J$-holomorphic curves, it follows that $f(S)=S$, {\it i.e.}, the sphere $S$ is invariant under the action. Therefore, $S$ can be blown down $\mathbb{Z}_k$-equivariantly, denote this by $\pi: M  \rightarrow M'$. Moreover, recall that $H_{2}(M',\mathbb{R})$ is isomorphic via pull-back to a subspace $H_{2}(M,\mathbb{R})$ of the blow-up so the induced action  on $H_{2}(M',\mathbb{R})$ is trivial. Moreover, since $\pi$ induces an isomorphism on $H_{1}$, the induced action on $H_{1}(M',\mathbb{R})$ is trivial. Therefore by Remark \ref{h12triv} the induced action on $M'$ is homologically trivial.

Repeating this process a finite number of times reduces the second Betti number to $b_{2}=2$. At that point, Theorem \ref{theorem} applies, completing the proof.
\end{proof}

Finally, we remark that the results presented in this section rely on the assumption that  the base genus $g\ge 2$. On the other hand, for geometrically ruled symplectic $4$-manifolds with base genus $g=0$, there is a uniform positive answer to Question \ref{kedra} for cyclic action of any order \cite{CKr}. 

When the base has genus one, Example \ref{torusexample} provides homologically trivial symplectic cyclic actions that do not extend to Hamiltonian circle actions.

\section{Algebraic Examples}\label{agreal}
In this section, we collect various examples of algebraic finite group actions on irrational ruled surfaces, which exhibit different aspects of our results. Let $C$ be a complex smooth projective curve. A smooth complex projective surface is called geometrically ruled if there is a morphism $\pi: S \rightarrow C$ such that every fiber is isomorphic to $\mathbb{P}^1$. Every geometrically ruled surface is isomorphic to $\mathbb{P}(E)$ for some rank $2$ holomorphic vector bundle over $C$; see Section \ref{prel}.

\subsection{Actions preserving $H_1$ but not $H_2$}
Let $\Sigma$ be any Riemann surface, and set $\alpha = e^{\frac{2 \pi i }{k}}$. Define a $\mathbb{Z}_{k}$-action on $M = \Sigma \times \mathbb{CP}^1$ by $$\alpha^{n}\cdot (p,[z_0:z_1]) = (p, [\alpha^{n} z_0:z_1])  .$$ This gives a homologically trivial (in fact, Hamiltonian) symplectic action with respect to the product symplectic form on $M$.

Now fix a point $p \in \Sigma$, and consider the free orbit $$\mathcal O =\{ (p,[\alpha^n:1])| n= 1, \cdots, k \}.$$ Blowing up the points in $\mathcal O$ using the standard equivariant local model yields a new symplectic $4$-manifold $\tilde{M} = Bl_{\mathcal O}(M)$, on which $\mathbb Z_k$ continues to act effectively. The resulting action is trivial on $H_{1}(\tilde{M},\mathbb{Z})$ but non-trivial on  $H_{2}(\tilde{M},\mathbb{Z})$. 

Furthermore, since the original action is holomorphic and admits fixed points, the blowup inherits a natural $\mathbb{Z}_k$-invariant  K\"{a}hler form. This can be realized explicitly via an equivariant embedding into a projective space followed by equivariant resolution of singularities. 

This gives the following example.

\begin{example} \label{h1noth2}
For every $k\in \mathbb Z_{>0}$ and every base genus, there exist effective symplectic $\mathbb{Z}_{k}$-actions on irrational ruled symplectic $4$-manifolds $M$ that act trivially on $H_{1}(M,\mathbb{Z})$ but non-trivially on  $H_{2}(M,\mathbb{Z})$. 
\end{example} 

\subsection{Short exact sequences of automorphisms}

The following example illustrates the short exact sequences  constructed in Section \ref{finclass}.

\begin{example}\label{nonsplit} Let $C \subset \mathbb{CP}^2$ be the smooth elliptic curve defined by $$z_{1}^3 + z_{2}^3 + z_{3}^3=0.$$ Let $G = S_{3}$ act on $C$ by permuting the homogeneous coordinates. Let $N_{C} \cong \mathcal{O}(3)|_{C}$ be the normal bundle of $C$. Then the action of $G$ on $C$ induces an action on $N_{C}$, and thus on the projectivized bundle $\mathbb{P}(N_{C}\oplus \mathcal{O})$, which is a non-trivial $S^2$-bundle over $C\cong T^2$.
\end{example}
 
Observe that every non-trivial element of $S_{3}$ induces a non-trivial automorphism of $C$, so $H_{1}$ is the trivial group. Moreover, the elements of order $3$ act on $C$ without  fixed points, and hence trivially on $H_{1}(C,\mathbb{Z})$; while the elements of order $2$ each fix a single point of $C$. 

For example, the transformation $\tau:[z_1:z_2:z_3] \mapsto [z_2:z_1:z_3]$ has four fixed points: $$ [1:1 : -\sqrt[3]{2} \alpha] \in C,$$ where $\alpha^3=1$ and $[-1:1:0]$. That is, it is the Galois involution giving the elliptic curve as a branched double cover of the Riemann sphere. By the Lefschetz fixed point theorem, $\tau$ acts on $H_{1}(C,\mathbb{Z})$ with trace $-2$ and order $2$, that is, $\tau_{*} = -Id$.

Next, consider the matrix
\[
  \begin{bmatrix}
    0 & 0 & 1 \\
    1 & 0 & 0 \\
   0 & 1 & 0
  \end{bmatrix}
\]
which has three distinct eigenvalues. The corresponding transformation $$\mu:[z_1:z_2:z_3] \mapsto [z_3:z_1:z_2]$$ on $\mathbb{CP}^2$ has three fixed points $[1:1:1], [1:\beta:\beta^2], [1:\beta^2:\beta],$ where $\beta = e^{\frac{2 \pi i}{3}}$. A direct check via the defining equation shows that none of these points lie on $C$. 
Therefore, the first exact sequence of Corollary \ref{finsym} takes the form $$1 \rightarrow \mathbb{Z}_{3} \rightarrow G \rightarrow \mathbb{Z}_2 \rightarrow 1.$$ In particular, this shows that the  short exact sequences in Corollary \ref{finsym} do not split in general. Moreover, since all of the transformations preserve both the $S^2$-fibers and the section $\mathbb{P}(N_{C} \oplus 0)$, they act trivially on $H_{2}(M,\mathbb{Z})$. This conclusion also follows directly from Proposition \ref{acth2}, as minimal irrational ruled surfaces, such as  $M=\mathbb{P}(N_{C}\oplus \mathcal{O})$, contain no Lagrangian spheres.

\subsection{Exotic dihedral actions preserving a complex structure} \label{dihedral}

From the structure of their fixed point sets and by Lemma \ref{holfix}, it follows that the homologically trivial actions of $\mathbb{Z}_2$ and $\mathbb{Z}_2 \oplus \mathbb{Z}_2$ constructed in Theorem \ref{exampleir} and Theorem \ref{exoticklein} respectively cannot preserve the product complex structure on $\Sigma\times S^2$. Furthermore, as shown, these actions do not extend to  Hamiltonian actions of $S^1$ or SO$(3)$, respectively. 

In this subsection, we observe that certain holomorphic involutions with respect to the integrable complex structures of a certain type on $\Sigma \times S^2$ can also have connected bisection fixed point sets. This provides an  algebro-geometric model for the $\mathbb Z_2$-action construction in Section \ref{exoticsymplecticinv}.

 In the following, we work over the field of complex numbers $\mathbb C$, and we adopt the notation established in Section \ref{aut} of the preliminaries.

\begin{prop} \label{algebraicrel} 
Let $C$ be a smooth algebraic curve, and $L$ a non-trivial line bundle on $C$ with $L \otimes L \cong \mathcal{O}$. Define $E =\mathcal{O} \oplus L$, and let $S = \mathbb{P}(E)$ be the associated ruled surface over $C$. Let $s_{1},s_{2}$ be the sections corresponding to $ \mathbb{P}(\mathcal{O} \oplus 0)$, $ \mathbb{P}( 0 \oplus L)$, respectively. 

Then, there is a short exact sequence:
 $$1 \rightarrow \mathbb{C}^*  \xrightarrow[]{\varphi_1} \mathrm{Aut}_{C}(S) \xrightarrow[]{\varphi_2}  \mathbb{Z}_2 \rightarrow 1 .$$

For an element $g \in \mathrm{Aut}_{C}(S)$, the following hold:
\begin{itemize} 
\item[a)]  $\varphi_{2}(g) =1, \;\; g \neq 1 \iff $ the fixed point set of $g$ is  the disjoint union $s_{1} \cup s_2$.

\item[b)] $\varphi_{2}(g) \neq 1 \iff $ $g^2 =\mathrm{Id}_{S}$, the fixed point set of $g$ is a smooth irreducible bisection of $\pi: S\rightarrow C$, and $g(s_{i}) = s_{j}$ for $i \neq j$.
\end{itemize}
\end{prop}
\begin{proof}
By \cite[Lemma 2 (2)]{Ma}, we have $M(E) = N(E)=0$, and the only two maximal subbundles of $E$ are $\mathcal{O} \oplus 0$ and $ 0 \oplus L$. Thus, by Lemma \ref{intsec}, the only sections of $S=\mathbb P(E)$ with self-intersection zero are $s_{1}$ and $s_{2}$, and all other section have strictly positive self-intersection. 

Consider the short exact sequence from Lemma \ref{exact}: $$1 \rightarrow \mathrm{Aut}(E)/\mathbb{C}^*  \xrightarrow[]{\varphi_1} \mathrm{Aut}_{C}(S) \xrightarrow[]{\varphi_2}  \Delta \rightarrow 1 .$$  By \cite[Lemma 4(2)]{Ma}, $\Delta \cong \mathbb{Z}_2$.
Note that $$\det(E)^{-1} \otimes L^2 \cong L^{-1} \cong L,$$ which admits no global sections. Hence, by \cite[Theorem 1(3)]{Ma}, Aut$(E)/\mathbb{C}^* \cong \mathbb{C}^*$ (see the proof of \cite[Theorem 2]{Ma}). The induced $\mathbb{C}^*$-action on  $S$ is given by $$z\cdot [(x,y)] = [(zx,y)] ,$$ and  fixes precisely the union $s_{1} \cup s_{2}$. Therefore,  all non-trivial elements of $\ker(\varphi_2) \cong \mathbb{C}^*$ have fixed point set $s_{1} \cup s_{2}$. 

Now let $g\in \mathrm{Aut}_{C}(S) $ be an automorphism with $\varphi_2(g)\ne 1$. We claim that $g^2=\mathrm{id}_S$. Since $\varphi_2(g^2)=\varphi_2(g)^2=1$, it follows that $g^2\in \ker(\varphi_2)$. Suppose for contradiction that $g^2\ne \mathrm{id}_S$, so that $g^2\in \ker(\varphi_2)\setminus \{\mathrm{id}_S\}$. Then Fix$(g^2)=s_1\cup s_2$. 

On each fiber, the action of $g$ is a M\"obius transformation of order $2$, hence an involution with exactly two fixed points.  Since M\"{o}bius transformations have the property that the fixed point set of $g$ is contained in that  of $g^2$, and they are equal  when $g^2$ is non-trivial, we conclude that Fix$(g)=\mathrm{Fix}(g^2)=s_{1} \cup s_2$. 

Now consider the action of $g$  on $S \setminus s_{1}$, which is isomorphic to the total space of the line bundle $L$. On this space, $g$ must act fiberwise by  M\"{o}bius transformations and must preserve the zero section, since $g$ fixes both $s_1$ and $s_2$.  But this implies that $g\in\varphi_{1}(\mathbb{C}^*) = \ker(\varphi_2)$, contradicting our assumption that $\varphi_2(g)\ne 1$.  Therefore, $g^2=\mathrm{id}_S$, as claimed.

The same argument shows that any element $g\notin \ker(\varphi_2)$ cannot fix the union $s_{1} \cup s_{2}$. Since $s_{1}$ and $s_{2}$ are the only sections of $S$ with self-intersection zero, and $g$ is an automorphism over $C$, the only possibility is that $g(s_{1}) = s_{2}$ and $g(s_{2}) = s_{1}$. 

Next, we examine the fixed point set Fix$(g)$. Since $g$ acts fiberwise as an involutive M\"obius transformation, each fiber contains exactly two fixed points under $g$. Thus, Fix$(g)$ is a curve in $S$ that intersects each fiber transversely in two points and is disjoint from $s_{1} \cup s_{2}$. 

Suppose for  contradiction that Fix$(g)$ is reducible and consists of two disjoint irreducible  components. Then each component must be a section, distinct from $s_1$ and $s_2$, and hence must have strictly positive self-intersection. However,  by the light cone lemma, such sections would necessarily intersect one of $s_{1}$ or $s_2$, contradicting the assumption that Fix$(g)$ is disjoint from $s_{1} \cup s_{2}$. Therefore, Fix$(g)$ must be a smooth irreducible bisection.
\end{proof}

The group structure of Aut$_C(S)$ is described in \cite[Theorem 2 (4)]{Ma} and is isomorphic to the subgroup $H\subset \mathrm{PGL}(2, \mathbb{C})$ given by $$H  = \left\{ 
   M_{\alpha}=
  \left[ {\begin{array}{cc}
   \alpha & 0 \\
   0 & 1 \\
  \end{array} } \right] | \alpha \in \mathbb{C}^*
\}  \cup  \{ 
   M'_{\beta}=
  \left[ {\begin{array}{cc}
   0 & \beta \\
   1 & 0 \\
  \end{array} } \right] | \beta \in \mathbb{C}^*
\right\}  .   $$

An element of $H$ lies in the kernel of $\varphi_2$ if and only if it is one of $M_{\alpha}$. 

For convenience, we record the group law as given in (\cite[Page 96, footnote]{Ma}):

$$M_{\alpha}M_{\alpha'} = M_{\alpha \alpha'} \;\;,\;\; 
M'_{\beta}M'_{\beta'} = M_{\beta \beta'^{-1}}    $$

$$M_{\alpha}M'_{\beta} = M'_{\alpha \beta} \;\;,\;\; 
M'_{\beta}M_{\alpha} = M'_{\beta \alpha^{-1}}    $$

We note that every element of the form $M'_{\beta}$ is an involution. The group $H$ is non-Abelian; for instance,  $M'_{-1}$ and $M_{\alpha}$ commute in PGL$(2, \mathbb{C})$ if and only if $\alpha = \pm1$. However, the subgroup consisting of the four matrices $$I_{2},\;\;\; M_{-1},\;\;\; M'_1,\;\;\; M'_{-1} $$ forms a Klein $4$-group isomorphic to $\mathbb{Z}_2 \times \mathbb{Z}_2$.

This  $\mathbb{Z}_2 \times \mathbb{Z}_2$ action may be generalized as follows:

\begin{example} \label{dihedralexample}
Consider the set: $$D =\left\{ M_{\alpha},M'_{\beta} \mid  \alpha^n=\beta^n=1\right\} \subset H.$$ It is not hard to check that $D$ is closed under the group operation, and therefore is a subgroup. Moreover, $D$ is isomorphic to a dihedral group of order $2n$. 
\end{example}

For the action of the above subgroup on $S$, the fixed point set of $n-1$ of its elements is a disjoint union of sections, while the fixed point set of the $n$  involutive elements $M'_{\beta}$ is an irreducible bisection.

Finally, we give an example of an action of the Klein $4$-group where every non-trivial element in $G = \mathbb{Z}_2 \times  \mathbb{Z}_2 $ fixes an irreducible bisection.

The exact sequence recalled in Lemma \ref{exact} will be important here. The following lemma gives a sufficient condition for an element $g \in \hbox{Aut}_{C}(\mathbb{P}(E))$ to be contained in the first term of this short exact sequence. Recall that a section of a geometrically ruled surface is a morphism $s: C \rightarrow \mathbb{P}(E)$ such that $\pi \circ s = Id_{C}$.

\begin{lemma}  \label{cyclic splitting}
Let $S = \mathbb{P}(E)$ be a geometrically ruled surface. Let $g \in Aut_{C}(S)$ be a finite order element such that Fix$(g)$ is a disjoint union of two sections. Then,  the following hold: \begin{enumerate}

\item $E$ splits as a sum of line bundles $E \cong L_{1} \oplus L_{2}$.

\item In the notation of short exact sequence of Lemma \ref{exact}, $g \in Aut(E)/\mathbb{C}^*$.

 \end{enumerate}
\end{lemma}
\begin{proof}
A holomorphic section $s: C \rightarrow \mathbb{P}(E)$ naturally defines a line subbundle $L \subset E$ \cite[Lemma 1.14]{Ma2}.  The two disjoint sections in Fix$(g)$ provide a splitting $E \cong L_{1} \oplus L_{2}$, such that Fix$(g) = [L_{1} \oplus \{0\}] \cup  [\{0\} \oplus L_{2}]$ (see \cite[Page 96, Footnote 1]{Ma}). Moreover, since $g$ acts by restricted M\"{o}bius transformations on $\mathbb{P}^1$ fibers fixing two points in the corresponding sections, it is contained in $Aut(E)/\mathbb{C}^*$. 
\end{proof}

The examples are obtained as an elementary transformation in three generic points of a product projective bundle over an elliptic curve $C$. Recall that for a geometrically ruled surface $\pi: S \rightarrow C$ and a point $q \in S$, the elementary transformation of $S$ in $q$ is a surface $S'$ obtained by blowing up $q$ and blowing down the strict transform of the $\mathbb{P}^1$-fiber containing $q$.

\begin{prop} \label{threebisection}
Let $C$ be an elliptic curve. Let $S$ be the geometrically ruled surface obtained from $C \times \mathbb{P}^1$ by performing elementary transformations at three points $q_1,q_2,q_3$, whose projections to both  $C$ and  $\mathbb{P}^1$ are pairwise distinct. Then, Aut$_{C}(S) \cong \mathbb{Z}_2 \times \mathbb{Z}_2$, and for every non-trivial element $g \in \mathbb{Z}_2 \times \mathbb{Z}_2$, the fixed locus Fix$(g)$ is an irreducible bisection. In particular, the involutions corresponding to the three non-trivial elements in $\mathbb{Z}_2 \times \mathbb{Z}_2$ do not extend to symplectic $S^1$-actions.
\end{prop}
\begin{proof}
Let $S = \mathbb{P}(E)$. By \cite[Theorem 4.7]{Ma2}, the bundle $E$ is indecomposable and $N(E)=1$.  By \cite[Theorem 3 (4)]{Ma}, one has Aut$_{C}(S) = \mathbb{Z}_2 \times \mathbb{Z}_2$.  For each non-trivial $g \in \mathbb{Z}_2 \times \mathbb{Z}_2$, the fixed point set Fix$(g)$ is either a union of two sections or an irreducible bisection. The former possibility is excluded by Lemma \ref{cyclic splitting}. Hence, Fix$(g)$ is a smooth irreducible bisection. Therefore, these involutions cannot extend to symplectic circle actions, because by  \cite[Proposition 5.1.7]{MS} the circle action must be Hamiltonian and the fixed point set must have at least two connected components.
\end{proof}

Considering a section $C \times \{p\} \subset C \times \mathbb{P}^1$, which is disjoint from $\{q_1,q_2,q_3\}$, the self-intersection of the strict transform in $S$ will be a section with self-intersection $3$. Therefore $S$ is topologically the non-trivial $S^2$-bundle over $C$.


\begin{thebibliography}{99}
\bibitem{A} M. S. Atallah, \textit{Remarks on symplectic circle actions, torsion and loops}, Algebr. Geom. Topol. 24 (2024), no. 4, 2367--2384.

\bibitem{AS} M. S. Atallah, E. Shelukhin, \textit{Hamiltonian no-torsion}, Geom. Topol. 27 (2023), no. 7, 2833--2897.

\bibitem{BGVC}  D.A.Baranov, V.Z.Grines, O.V.Pochinka, E.E.Chilina, \textit{On a Classification of Periodic Maps on the $2$-Torus}, Russian Journal of Nonlinear Dynamics, 2023, vol. 19, no. 1, pp. 91--110

\bibitem{BL}  O. Buse, J. Li, \textit{Symplectic isotopy on non-minimal ruled surfaces}, Math. Z. 304 (2023), no. 3, Paper No. 44, 24 pp.


\bibitem{Ca} C. S\'{a}ez Calvo, \textit{Finite groups acting on
smooth and symplectic $4$-manifolds}, PhD thesis. Universitat de Barcelona, 2019

\bibitem{Chen} W. Chen, \textit{Group actions on $4$-manifolds: some recent results and open questions}, Proceedings of the G\"okova Geometry-Topology Conference 2009, 1--21, Int. Press, Somerville, MA, 2010.

\bibitem{chenq} W. Chen, \textit{Finite symmetry in dimension $4$: ten open questions}, \url{https://people.math.umass.edu/~wchen/10-open-problems.pdf}.

\bibitem{CKw} W. Chen, S. Kwasik, \textit{Symplectic symmetries of $4$-manifolds}, Topology 46 (2007) 103--128.

\bibitem{CLW} W. Chen, T.J. Li, W. Wu, \textit{ Symplectic rational  $G$-surfaces and equivariant symplectic cones}, J. Differential Geom. 119 (2021), no. 2, 221--260.



\bibitem{CKr} R. Chiang and L. Kessler, \textit{Cyclic actions on rational ruled symplectic four-manifolds}, Transform. Groups 24 (2019), no. 4, 987--1000.

\bibitem{CK} R. Chiang and L. Kessler, \textit{Homologically trivial symplectic cyclic actions need not extend to Hamiltonian circle actions},  J. Topol. Anal. 12 (2020), no. 4, 1047--1071.

\bibitem{CKem} A. Constantin, B. Kolev,  \textit{The theorem of K\'er\'ekjart\`o on periodic homeomorphisms of the disc and the sphere}, Enseign. Math. (2) 40 (1994), no. 3-4, 193--204. 


\bibitem{D} J. Dydak, \textit{An addendum to the Vietoris-Begle Theorem}, Topology and its applications. 23 75--86 (1986).

\bibitem{FM} B. Farb and D. Margalit, \textit{A Primer on Mapping Class Groups}, Princeton Mathematical Series. PRINCETON UNIVERSITY PRESS (2012).

\bibitem{FH} J. Franks, M. Handel, \textit{Periodic points of Hamiltonian surface diffeomorphisms}, Geom. Topol. 7 (2003), 713--756.



\bibitem{GGbey} V. L. Ginzburg, B. G\"urel, \textit{The Conley conjecture and beyond}, Arnold Math J., 1 (2015), 299--337.


\bibitem{Gom} R. E. Gompf, 
\textit{A New Construction of Symplectic Manifolds}, Annals of Mathematics, Vol. 142, No. 3 (Nov., 1995), pp. 527--595.

\bibitem{Gr}A. Grothendieck, \textit{G\'eom\'etrie formelle et g\'eom\'etrie alg\'ebrique}, S\'eminaire Bourbaki : ann\'ees 1958/59 - 1959/60, expos\'es 169--204, S\'eminaire Bourbaki, no. 5 (1960), Talk no. 182, 28 p.



\bibitem{GL} N. Guelman, I. Liousse, \textit{Burnside problem for groups of homeomorphisms of compact surfaces}, Bull. Braz. Math. Soc. (N.S.) 48 (2017), no. 3, 389--397. 

\bibitem{H} A. Hatcher, \textit{Algebraic Topology}, Cambridge University Press (2002).

\bibitem{JS} W. Jaco and P.B. Shalen, \textit{Surface homeomorphisms and periodicity}, Topology Volume 16, Issue 4, 1977, Pages 347--367.


\bibitem{JO} K. J\"{a}nnich, E. Ossa, \textit{On the signature of an involution}, Topology
Volume 8, Issue 1, January 1969, Pages 27--30.

\bibitem{Kar} Y. Karshon, \textit{Periodic Hamiltonian flows on four dimensional manifolds}, Memoirs of the Amer. Math. Soc. 672 (1999).

\bibitem{KS} S. Kwasik, R. Schultz, \textit{Homological properties of periodic homeomorphisms of $4$-manifolds}, Duke Math. J. 58 (1989), no. 1, 241--250.

\bibitem{L} J. M. Lee, \textit{Introduction to Smooth Manifolds} (Second Edition),  Graduate Texts in Mathematics, 218, Springer. 

\bibitem{Lh} H. Li, \textit{$\pi_1$ of Hamiltonian $S^1$ manifolds}, Proc. Amer. Math. Soc. 131 (2003), no. 11, 3579--3582.


\bibitem{LLW} J. Li, T. J. Li, W. Wu, \textit{Symplectic Torelli groups of rational surfaces}, arXiv:2212.01873.

\bibitem{LL}  T. J. Li, A.-K. Liu, \textit{Uniqueness of symplectic canonical class, surface cone and symplectic cone of $4$-manifolds with  $B^+=1$},  J. Differential Geom. 58 (2001), no. 2, 331--370.

\bibitem{LW}  T. J. Li, W. Wu, \textit{Lagrangian spheres, symplectic surfaces and the symplectic mapping class group},
Geom. Topol. 16 (2012), no. 2, 1121--1169.


\bibitem{LZ}  T.J., Li, W. Zhang, \textit{ $J$-holomorphic curves in a nef class}, Int. Math. Res. Not. 2015, 12070--12104
(2015).

\bibitem{Ma} M. Maruyama, \textit{On automorphism groups
of ruled surfaces}, J. Math. Kyoto Univ. (JMKYAZ)
11-1 (1971) 89--112.



\bibitem{Ma2} M. Maruyama, \textit{On classification of ruled surfaces}, Lect. in Math., Dept.
 Math. Kyoto Univ. 3, Kinokuniya, Tokyo (1970).

\bibitem{Mciso} D. McDuff,  \textit{From symplectic deformation to isotopy}, First Int. Press Lect. Ser., I
International Press, Cambridge, MA, 1998, 85--99.

\bibitem{Mcsacs} D. McDuff, \textit{Symplectomorphism groups and almost complex structures}, Monogr. Enseign. Math., 38[Monographs of L'Enseignement Math\'ematique]
L'Enseignement Math\'ematique, Geneva, 2001, 527--556.

\bibitem{Mc} D. McDuff, \textit{Hamiltonian $S^1$--manifolds are uniruled}, Duke Math. J. 146 (2009), no. 3, 449--507.

\bibitem{MS} D. McDuff, D. Salamon, \textit{Introduction to Symplectic Topology}, Third Edition. Oxf. Grad. Texts Math.
Oxford University Press, Oxford, 2017, xi+623 pp.

\bibitem{MN} S. Mukai and Y. Namikawa, \textit{Automorphisms of Enriques surfaces which act trivially on the cohomology groups}, Invent. Math. 77 (1984), 383--398.


\bibitem{Mu} I. Mundet i Riera, \textit{Finite
groups acting symplectically on $T^2 \times S^2$}, Trans. Amer. Math. Soc. 369 (2017), no. 6,
4457--4483.



 \bibitem{P} R.S. Palais, \textit{On the existence of slices for actions of noncompact Lie groups}, Ann. of Math. (2) 73 (1961)295--323.

\bibitem{Pol} L Polterovich, \textit{Growth of maps, distortion in groups and symplectic geometry}, Invent. Math. 150 (2002) 655--686.

\bibitem{R} Y. Rong, \textit{Degree one maps between geometric $3$-manifolds}, Trans. Amer. Math. Soc. 332 (1992), no. 1, 411--436.

\bibitem{SS} N. Sheridan, I. Smith, \textit{Symplectic topology of K3 surfaces via mirror symmetry}, 
J. Amer. Math. Soc. 33 (2020), no.3, 875--915.

\bibitem{SSell} V. Shevchishin, G. Smirnov, \textit{Elliptic diffeomorphisms of symplectic $4$-manifolds}, J. Symplectic Geom. 18 (2020), no. 5, 1247--1283.

\bibitem{T} W.P. Thurston,  \textit{On the geometry and dynamics of diffeomorphisms of surfaces}, American Mathematical Society. Bulletin. New Series, 19(2) 417--431, (1988). 

 \bibitem{VN} A.Ju. Volovikov and Nguen Le Anh, \textit{On the Vietoris-Begle Theorem}, Vestnik Moskov. Univ. Ser. 
I Math. Mekh. 1984 no. 3, 70--71.

\bibitem{W} R. Walczak, \textit{Existence of Symplectic Structures on Torus Bundles Over Surfaces}, Annals of Global Analysis and Geometry. Volume 28, pages 211--231, (2005).


\bibitem{Zh} W. Zhang, \textit{Moduli space of $J$-holomorphic subvarieties}, Selecta Mathematica. 
Volume 27, article number 29, (2021).   
\end{thebibliography}
\end{document}